\documentclass[12pt]{amsart}
\usepackage{amsmath}
\usepackage{amssymb}
\usepackage{amsfonts}
\usepackage{mathrsfs}
\usepackage{graphicx}
\usepackage{epsfig}
\usepackage[T1]{fontenc}
\numberwithin{equation}{section}       % Number formulas within sections

\theoremstyle{plain}
\newtheorem{theorem}{Theorem}[section]

\newtheorem{prop}{Proposition}[section]

\newtheorem{coro}[prop]{Corollary}
\newtheorem{lemma}[prop]{Lemma}

\theoremstyle{definition}
\newtheorem{definition}[prop]{Definition}

\theoremstyle{remark}
\newtheorem{remark}[prop]{Remark}

\newtheoremstyle{citing}% name
  {3pt}%      Space above, empty = `usual value'
  {3pt}%      Space below
  {\itshape}% Body font
  {}%         Indent amount (empty = no indent, \parindent = para indent)
  {\bfseries}% Thm head font
  {.}%        Punctuation after thm head
  {.5em}%     Space after thm head: " " = normal interword space;
  {\thmnote{#3}}% Thm head spec

\theoremstyle{citing}
\DeclareMathAlphabet{\mathpzc}{OT1}{pzc}{m}{it} % Zapf Chancery math alphabet

\newcommand{\C}{\mathbb{C}}

\newcommand{\N}{\mathbb{N}}

\newcommand{\R}{\mathbb{R}}

\newcommand{\Z}{\mathbb{Z}}

\newcommand{\teta}{\widetilde{\teta}}

\newcommand{\dist}{d}

\DeclareMathOperator{\diam}{diam}

\DeclareMathOperator{\supp}{supp} % Support of a measure
\begin{document}

\title[]{On invariant measures of "satellite" infinitely renormalizable quadratic polynomials}

\author{Genadi Levin}

\address{Institute of Mathematics, The Hebrew University of Jerusalem, Givat Ram,
Jerusalem, 91904, Israel}

\email{levin@math.huji.ac.il}

\author{Feliks Przytycki}

\address{Institute of Mathematics, Polish Academy of Sciences, \'Sniadeckich St., 8, 00-956 Warsaw,
Poland}

\email{feliksp@impan.pl}

\thanks{The first author partially supported by ISF grant 1226/17, Israel.
The second author partially supported by National Science Centre, Poland, Grant OPUS 21 "Holomorphic dynamics, fractals, thermodynamic formalism”, 2021/41/B/ST1/00461.}

%\author{Weixiao Shen}

%\address{Department of Mathematics, National University of Sinapore,
%10 Lower Kent Ridge Road, Singapore 119076}
%\address{Shanghai Center for Mathematical Sciences, Fudan University, 220 Handan Road, Shanghai, China 200433}

%\email{wxshen@fudan.edu.cn}

\date{\today}

%\author{Genadi Levin, Feliks Przytycki and Weixiao Shen}

%\title[]{Upper Lyapunov exponent}
%\author{}

%\date{\today}

\maketitle

\begin{abstract}
Let $f(z)=z^2+c$ be an infinitely renormalizable quadratic polynomial and $J_\infty$ be the intersection of forward orbits of "small" Julia sets
of its simple renormalizations. We prove that if $f$ admits an infinite sequence of satellite renormalizations, then every invariant measure of $f: J_\infty\to J_\infty$
is supported on the postcritical set and has zero Lyapunov exponent.
%$J_\infty$ carries no $f$-invariant measures with the positive Lyapunov exponent.
Coupled with \cite{LPS}, this implies that the Lyapunov exponent of such $f$ at $c$ is equal to zero, which answers partly a question posed by Weixiao Shen.
\end{abstract}

\section{Introduction}

We consider the dynamics $f: \C\to \C$ of a quadratic polynomial. Up to a linear change of coordinates, $f$ has the form $f_c(z)=z^2+c$ for some $c\in\C$.
%Let $f: \C\to \C$ be a dynamical system on the complex plane generated by a quadratic polynomial $f$, i.e., up to an affine change of coordinates, $f$ is of the form $f_c(z)=z^2+c$ for some $c\in\C$.
In this paper, which is the sequel of \cite{LPI},
we assume that $f$ is infinitely-renormalizable. Moreover, in the main results we assume that $f$ has infinitely many "satellite renormalizations", see e.g. \cite{mcm}, or below for definitions.
Dynamics, geometry and topology of such system can be very non-trivial,
in particular, due to the fact that different renormalization levels are largely independent.

Historically, the first example of infinitely-renormalizsable one-dimensional map was, probably, the Feigenbaum period-doubling quadratic polynomial $f_{c_F}$, where $c_F=-1.4...$ \cite{feig}.
%Note that all the renormalizations of $f_{c_F}$ are satellite.
The Julia set of $f_{c_F}$ is locally connected \cite{HuJi} as it follows from so-called "complex bounds",
a compactness property of renormalizations. This is a key tool since \cite{Su}, in particular, in proving the Feigenbaum-Coullet-Tresser universality conjecture \cite{Su, mcm1, lu}.
Perhaps, more striking for us
are Douady-Hubbard's examples, or alike, of infinitely-renormalizable quadratic polynomials with non-locally connected Julia sets \cite{Mi0, So, L, L1, L1add, CS, CP}. As for the Feigenbaum polynomial $f_{c_F}$,
all the renormalizations of such maps are satellite, although, contrary to $f_{c_F}$, combinatorics is unbounded (which, in turn, implies that those maps cannot have complex bounds \cite{BLOT}).

%On a more conceptual level, 'Density of Hyperbolicity Conjecture (DHC)' (called also the Fatou conjecture) asserts that any rational map (polynomial) can be approximated by hyperbolic rational maps (polynomials) of the same degree. In polynomials of degree $2$, there is a stronger conjecture called MLC(='Mandelbrot set is Locally Connected').
%After a breakthrough work of Yoccoz \cite{H} on the MLC, the only obstacle for proving DHC for quadratic polynomials are infinitely-renormalizable ones.

Dynamics of every holomorphic endomorphism of the Riemann sphere $g:\hat\C\to\hat\C$ classically splits $\hat\C$ into two subsets: the Fatou set $F(g)$ and its complement the Julia set $J(g)$, where
$F(g)$ is the maximal (possibly, empty) open set where the sequence of iterates $g^n$, $n=0,1,...$ forms a normal (i.e., a precompact) family.
%Then $J(g)$ is the closure of all repelling cycles of the system.
See e.g. \cite{CG}, \cite{Mib} for the Fatou-Julia theory and \cite{Mary} for a recent survey.

%Given a rational map $g:\hat\C\to\hat\C$ of the Riemann sphere of degree $d\ge 2$, the Fatou set $F(g)$ of $g$ is the maximal open set where the sequence of iterates $g^n$, $n=0,1,...$ form a normal
%(i.e., a precompact) family. Its complement $J(g)$ is the Julia set of $g$, see e.g. \cite{CG}, \cite{Mib} for the background and \cite{Mary} for a recent survey.

If $g$ is a polynomial, then the Julia set $J(g)$ coincides with the boundary of the basin of infinity $A(\infty)=\{z\in\C| \lim_{n\to\infty}g^n(z)=\infty\}$ of $g$.
The complement $\C\setminus A(g)$ is called the filled Julia set $K(g)$ of the polynomial $g$. The compact $K(g)\subset \C$ is connected if and only if it contains all critical points
of $g$ in the complex plane.

A quadratic polynomial $f_c$ with connected filled Julia set $K(f)$ is {\it renormalizable} if, for some topological disks $U\Subset V$  around the critical point $0$ of $f_c$,
and some $p\ge 2$ (period of the renormalization), the restriction $F:=f_c^p:U\to V$ is a proper branched covering map (called polynomial-like map) of degree $2$
%polynomial-like (PL) mapping of degree $2$
and the non-escaping set $K(F)=\{z\in U: F^{n}(z)\in U \mbox{ for all } n\ge 1\}$ (called the filled Julia set of the polynomial-like map $F$) is connected.
%\{z| F^n(z)\in U, n=0,1,...\}$.
%Here, PL map is a proper covering map of degree bigger than $1$.
The map $F:U\to V$ is then a {\it renormalization} of $f_c$
and the set $K(F)$ is a {\it "small" (filled) Julia set} of $f_c$.
By the theory of polynomial-like mappings \cite {DH}, there is a quasiconformal homeomorphism of $\C$, which is conformal on $K(F)$, that
conjugates $F$ on a neighborhood of $K(F)$ to a uniquely defined another quadratic polynomial $f_{c'}$ with connected filled Julia set.
%is hybrid equivalent, in particulat, conjugate to a unique defined another quadratic polynomial $f_{c'}$ with connected filled Julia set.
%(see \cite{DH} for exact definitions and the theory of polynomial-like mappings).
%By \cite{DH}, this is equivalent to the condition that $g:=f^p: U\to V$ is a covering map with a single critical point at $0$ and
%the Julia set $J_g=\{z\in U: f_c^{pn}(z)\in U \mbox{ for all } n\ge 1\}$ of $g$ is connected.
%The map $f_c^p:U\to V$ is then a {\it renormalization} of $f_c$
%and the set $K(F)=\{z\in U: f_c^{pn}(z)\in U \mbox{ for all } n\ge 1\}$ is a {\it "small" (filled) Julia set} of $f_c$.
If $f_{c'}$ is renormalizable by itself, then $f_c$ is called twice renormalizable, etc.
If $f_c$ admits infinitely many renormalizations, it is called {\it infinitely-renormalizable}.
The renormalization $F=f^p_c$ is {\it simple} if any two sets $f^i(K(f))$, $f^j(K(F))$, $0\le i<j\le p-1$,
are either disjoint or intersect each other at a unique point which does not separate either of them.
A simple renormalization $f^{p_n}$ is called {\it primitive} if all sets $f^i(K_n)$, $i=0,\cdots, p_n-1$, are disjoints and {\it satellite} otherwise.

%In spite of many efforts and great developments, not much is known about the topology and dynamics
%of an arbitrary infinitely-renormalizable quadratic polynomial.

To state our main results, Theorems \ref{thm:meas+}, let $f(z)=z^2+c$ be infinitely renormalizable. Then its Julia set $J=J(f)$ coincides with the filled Julia set $K(f)$ and is a nowhere dense compact full connected subset of $\C$.
Let $1=p_0<p_1<...<p_n<...$ be the sequence
of consecutive periods of simple renormalizations of $f$ and $J_n\ni 0$ denote the "small" Julia set of the
$n$-renormalization (where $J_0=J$). Then $p_{n+1}/p_n$ is an integer, $f^{p_n}(J_n)=J_n$, for any $n$, and
%$\{J_n\}_{n=1}^\infty$ is a strictly decreasing sequence of continua without interior, all containing $0$.
%which contains $0$
%Let
$f$-orbits of $J_n$,
$$orb(J_n)=\cup_{j\ge 0}f^j(J_n)=\cup_{j=0}^{p_n-1}f^j(J_n),$$
$n=0,1,...$, form a strictly decreasing sequence of compact subsets of $\C$.
%be the $f$-orbit of $J_n$ and
Let
$$J_\infty=\cap_{n\ge 0}orb(J_n)$$
be the intersection of the orbits of the "small" Julia sets $J_n$.
%(note that $\{orb(J_n)\}_{n=1}^\infty$ is a decreasing sequence of compact subsets of $\C$).
For every $n$, repelling periodic orbits of $f$ are dense in $orb(J_n)$,
while each component of $J_\infty$ is wandering.
%,hence
In particular, $J_\infty$ contains no periodic points of $f$.

%$$J_\infty=\cap_{n\ge 0}\cup_{j=0}^{p_n-1}f^j(J_n)$$
%be the intersection of orbits of the "small" Julia sets.
%$J_\infty$ is a compact $f$-invariant set
%which contains the omega-limit set $\omega(0)$ of $0$.
%and each component of $J_\infty$ is wandering, in particular, $J_\infty$ contains no periodic orbits of $f$.
%Recall \cite{mcm} that a renormalization $f^{p_n}$, $n\ge 1$, is {\it simple} if any two sets $f^i(J_n)$, $f^j(J_n)$, $0\le i<j\le p_n-1$,
%are either disjoint or intersect each other at a unique point which does not separate either of them.
%A simple renormalization $f^{p_n}$ is called {\it primitive} if all sets $f^i(J_n)$, $i=0,\cdots, p_n-1$, are disjoints and {\it satellite} otherwise.
%Let
%$$orb(J_n)=\cup_{j\ge 0}f^j(J_n)=\cup_{j=0}^{p_n-1}f^j(J_n).$$
%Note that, for every $n$, repelling periodic orbits of $f$ are dense in $orb(J_n)$
%while each component of $J_\infty$ is wandering,
%in particular, $J_\infty$ contains no periodic points of $f$.
Let
$$P=\overline{\{f^n(0)| n=1,2,...\}}$$
be the postcritical set of $f$. Clearly,
$$P\subset J_\infty.$$
Moreover, the critical point $0$ is recurrent, hence,
$$P=\omega(0),$$
where $\omega(z)$ is the omega-limit set of a point $z\in J$.

We prove in \cite{LPI} that $J_\infty$ cannot contain any hyperbolic set.
%As a consequence, the postcritical set $P$ must intersect the omega-limit set $\omega(x)$ of each $x\in J_\infty$.
On the other hand, a hyperbolic set of a rational map always carries an invariant measure with a positive Lyapunov exponent.
So a generalization of \cite{LPI} would be that $J_\infty$ never carries such a measure.
%The affirmative answer to the following conjecture thus
%would be a far reaching generalization of Theorem \ref{thm:exp}:
%{\it For an infinitely renormalizable $f(z)=z^2+c$, the set $J_\infty$ carries no an invariant measure with a positive Lyapunov exponent.}
Here we prove this generalization for a class of "satellite" infinitely-renormalizable quadratic polynomials:
%$f$, that is, such that no renormalization of $f$ is conjugate to
%a quadratic polynomial having only primitive renormalizations:

\begin{theorem}\label{thm:meas+}
Suppose that $f(z)=z^2+c$ admits infinitely many satellite renormalizations.
Then $f: J_\infty\to J_\infty$ has no invariant probability measure with positive Lyapunov exponent.
\end{theorem}

\begin{remark}\label{r:gen}
Conjecturally, the same conclusion should hold for any infinitely-renormalizable $f(z)=z^2+c$.
One can show this assuming that the Julia set of $f$ is locally-connected
(e.g., this is the case if $f$ admits complex bounds). Indeed, if $f:J_\infty\to J_\infty$ had
an invariant probability measure with positive Lyapunov exponent, then, taking a typical point of this measure and repeating the proof of \cite{LPS}, Corollary 5.5, we would conclude that the Julia set of $f$ is not locally-connected (in fact, $J_\infty$ contains a non-trivial continuum). Thus the only open case remains when
$f$ has only finitely many satellite renormalizations and $J_\infty$ contains a non-trivial continuum.
%(e.g., this is the case if $f$ admits complex bounds), this conjecture does hold, even in a much more general situation.
%Namely, let $g$ be a polynomial which admits infinitely many simple renormalizations around a critical value $v$ of $g$.
%Then the set $J_\infty(v)$ corresponding to the sequence of infinite renormalizations around $v$ can be similar defined.
%Now, if $g: J_\infty(v)\to J_\infty(v)$ has an invariant probability measure with positive Lyapunov exponent, then, taking a typical point of this measure and repeating the proof of \cite{LPS}, Corollary 5.5, we conclude that the Julia set of $g$ would be not locally-connected.
%It is proved in \cite{LPS}, Corollary 5.5, that if the upper Lyapunov exponent $\hi_+(g, v)$ of $g$ at $v$ is positive, then the Julia set
%of $g$ is not locally-connected.
\end{remark}
\begin{remark} For every rational map $f:\C\to\C$ (in particular, quadratic polynomial) and every invariant probability measure supported on Julia set Lyapunov exponents are non-negative, see \cite{Pr0} (compare a remark preceding Corollary \ref{coro:J}). On the other hand, if $f$ is hyperbolic or  non-uniformly hyperbolic  (topologically Collet-Eckmann) Lyapunov exponents for all invariant probability measures supported on Julia set are positive and bounded away from $0$, see \cite{P-ICM}.
\end{remark}
Let us comment on the behavior of the restriction map $f: J_\infty\to J_\infty$ where $f$ as in Theorem \ref{thm:meas+}. First, by \cite{LPI}, the postcritical set $P$ must intersect the omega-limit set $\omega(x)$ of each $x\in J_\infty$. At the same time, dynamics and topology of the further restriction $f: P\to P$ can vary. Indeed, there are
infinitely renormalizable quadratic polynomials $f$ with all renormalizations being of satellite type such that at least one of the following holds\footnote{A more complete description of $f:P\to P$ should follow from the methods developed in \cite{CP}.}:
%\begin{itemize}
%\item

(1) $f:P\to P$ is not minimal. This case happens in Douady-Hubbard's type examples. Indeed, by the basic construction \cite{Mi0}, $J_\infty$ then contains a closed invariant set $X$ (which is the limit set for the collection of $\alpha$-fixed points of renormalizations)
such that $0\notin X$. By \cite{LPI}, $X\cap P$ is non-empty. Thus $X\cap P$ is an invariant non-empty proper compact subset of $P$.
%\item

(2) $P$ is a so-called "hairy" Cantor set, in particular, $P$ contains uncountably many non-trivial continua. This case takes place following \cite{CP}.
%\item

(3) $P$ is a Cantor set and $f:P\to P$ is minimal; this happens whenever $f$ either admits complex bounds
%(a compactness property of renormalizations intensively studies since \cite{Su})
(which then imply $J_\infty=P$)
or
%the set $P$ has a bounded geometry in its representation through the renormalizations
is robust
%(which means, roughly, a bounded geometry of the postcritical set
\cite{mcm}\footnote{The "robustness" can happen without "complex bounds" as it follows from \cite{CP} combined with \cite{BLOT}.}.
Under either of the two conditions,
$f: P\to P$ is a minimal homeomorphism, which is topologically conjugate to $x\mapsto x+1$
acting on the projective limit of the sequence of groups $\{\Z/p_n\Z\}_{n=1}^\infty$; in particular, $f:P\to P$ (hence, also $f:J_\infty\to J_\infty$, as it follows from the next Corollary \ref{coro:J}) is uniquely ergodic in this case.
%\end{itemize}

%Let
%$$P=\overline{\{f^n(0)| n=1,2,...\}}$$ be the postcritical set of $f$. Clearly,
%$$P\subset J_\infty.$$
%Recall that $\omega(0)$, the omega-limit set of $0$, is the limit set of the postcritical set $P:=\{f^n(0)\}_{n\ge 1}$ of $f$.
%In the considered case of an infinitely renormalizable quadratic polynomial, $\omega(0)$ is the closure of the critical orbit and a subset of $J_\infty$.

\

Theorem \ref{thm:meas+} yields the following dichotomy about the measurable dynamics of $f: J \to J$ on the Julia set $J$ of $f$.
Recall that, by \cite{Pr0}, any invariant probability measure on the Julia set of a rational function has non-negative exponents.
\begin{coro}\label{coro:J}
Let $\mu$ be an invariant probability ergodic measure of $f: J\to J$. Then either
\begin{enumerate}
\item [(i)] $\supp(\mu)\cap J_\infty=\emptyset$ and its Lyapunov exponent $\chi(\mu)>0$,

or
\item [(ii)] $\supp(\mu)\subset P$ and $\chi(\mu)=0$.
\end{enumerate}
\end{coro}
In particular, the set $J_\infty\setminus P$ is "measure invisible", see also Proposition \ref{pr} which is a somewhat stronger version of Corollary \ref{coro:J}.

%\begin{theorem}\label{thm:meas+}
%Suppose that $f(z)=z^2+c$ admits infinitely many satellite renormalizations.
%Then $J_\infty$ carries no $f$-invariant probability measures with the positive Lyapunov exponent.
%\end{theorem}

%Let
%$$orb(J_n)=\cup_{j\ge 0}f^j(J_n)=\cup_{j=0}^{p_n-1}f^j(J_n).$$
%Note that, for every $n$, repelling periodic orbits of $f$ are dense in $orb(J_n)$
%while each component of $J_\infty$ is wandering.
%in particular, $J_\infty$ contains no periodic points of $f$.

%Given a holomorphic map $g$ in a neighborhood of its invariant set $Y$, let
%an $f$-invariant set $Y$, let
%$$\overline\chi_g(Y)=\sup\{\overline\chi_g(x)| x\in Y\}$$
%where $\overline\chi_g(x)=\limsup\frac{1}{m}\log|(g^m)'(x)|$, the upper Lyapunov exponent of $g$ at the point $x$.
%Clearly, $\overline\chi_g(Y_1)\le \overline\chi_g(Y_2)$ if $Y_1\subset Y_2$.
%In particular, $\overline\chi_f(orb(J_n))$, $n=1,2,...$, is a decreasing sequence of positive numbers.
%The next corollary answers partly a question of the third author that inspired the present work.
\begin{coro}\label{expat0}
If $f$ admits infinitely many satellite renormalizations, then
\begin{equation}\label{expany}
\limsup_{n\to \infty} \frac{1}{n}\log|(f^n)'(x)|\le 0 \mbox{ for any } x\in J_\infty,
\end{equation}
and
\begin{equation}\label{expcr}
\lim_{n\to \infty} \frac{1}{n}\log|(f^n)'(c)|=0.
\end{equation}
%and
%\begin{equation}\label{expset}
%\lim_{n\to\infty}\overline\chi_f(orb(J_n))=0.
%\end{equation}
%where $\chi_{n, per}$ is the supremum of Lyapunov exponents over all periodic orbits of the $n$-renormalization, i.e, those that are contained in $\cup_{i=0}^{p_n-1} J_n$.
\end{coro}

%The first conclusion (\ref{expcr}) answers partly a question asked by Weixiao Shen which inspired the present work as well as the prior \cite{LPI}.

%\begin{remark}\label{intro}
%The limit (\ref{expset}) can be rewritten as follows. First, notice that, given $n$, any map $f^{p_n}: f^j(J_n)\to f^j(J_n)$, $0<j<p_n$,
%which is defined and holomorphic in a neighborhood of $f^j(J_n)$, is holomorphically conjugate to $f^{p_n}: J_n\to J_n$. It follows,
%$$\overline\chi_{f^{p_n}}(f^j(J_n))=\overline\chi_{f^{p_n}}(J_n), \mbox{ for } j=1,2,\cdots, p_n-1.$$
%In turn, (\ref{expset}) is equivalent to:
%$$\overline\chi_{f^{p_n}}(J_n)=o(p_n) \mbox{ as } n\to\infty.$$
%\end{remark}
%\begin{proof}[Proof of Corollary \ref{expat0}]
%By \cite{LPS}, $\liminf_{n\to \infty} \frac{1}{n}\log|Df^n(c)|\ge 0$. On the other hand, if
%$\overline\chi(c)$ was strictly positive that would imply, by a standard argument, the existence
%of an $f$-invariant measure with positive exponent supported in $\omega(c)\subset J_\infty$,
%with a contradiction to Theorem \ref{thm:meas+}. This proves (\ref{expcr}).
%The proof of (\ref{expset}) is similar taking into account that $\overline\chi(orb(J_n))>0$ for all $n$.
%\end{proof}

%An immediate consequence is:
%\begin{coro}\label{c:main}
%Suppose that $f(z)=z^2+c$ admits infinitely many satellite renormalizations.
%Then any invariant measure of $f: J_\infty\to J_\infty$ is, in fact, supported on $\omega(0)$ and has a non-positive exponent.
%\end{coro}

For the proof of Corollaries \ref{coro:J}-\ref{expat0}, see Sect. \ref{proofJ}. The proof of Theorem \ref{thm:meas+} occupies sections \ref{prel}-\ref{thm:proof}.

As in \cite{LPI}, we use heavily a general result of \cite{Pr} on the accessibility although the main idea of the proof is different.
Indeed,
in \cite{LPI} we utilize the fact that the map cannot be one-to-one on an infinite hyperbolic set.
At the present paper, to prove Theorem \ref{thm:meas+} we assign, loosely speaking, an external ray to a typical point of a hypothetical measure with positive exponent
such that the family of such rays is invariant and has a controlled geometry.
Given a satellite renormalization $f^{p_n}$ we use the measure and the above family of rays to choose a point $x$ and
build a special domain that covers a "small" Julia set $J_{n,x}\ni x$ such that there is a univalent pullback of the domain by $f^{p_n}$ along the
renormalization that
enters into itself, leading to a contradiction. The choice of $x$ is 'probabilistic', i.e., made from sets of positive measure,
and the construction of the domain differs substantially depending on whether all satellite renormalizations of $f$ are doubling
or not.

%In the course of the proof we establish some general facts which can be of independent interests as e.g.
%Lemma \ref{p2}.

{\bf Acknowledgment.} The conclusion (\ref{expcr}) of Corollary \ref{expat0} that the Lyapunov exponent at the critical value equals zero answers partly a question by Weixiao Shen\footnote{ Shen asked the following question, in relation with Corollary 5.5 of \cite{LPS}: Is it possible that
the upper Lyapunov exponent at a critical value $v$ of a polynomial $g$ is positive assuming that $g$ is infinitely-renormalizable around $v$? See also Remark \ref{r:gen}}
which inspired the present work as well as the prior one \cite{LPI}. The authors thank the referee for careful reading the paper and many helpful comments.
%We thank Weixiao Shen for helpful comments.

\section{Preliminaries}\label{prel}
Here
%and in Section \ref{prelcont}
we collect, for further references and use throughout the paper, necessary notations and general facts. (A)-(D) are slightly adapted versions of (A)-(D) in Sect. 2, \cite{LPI}
which are either well-known \cite{mcm}, \cite{Mi1}, %follow readily from known ones,
or are proved here.

Let $f(z)=z^2+c$ be infinitely renormalizable. We keep the notations of the Introduction.
%further references necessary notations and general (combinatorial) facts about infinitely-renormalizable quadratic polynomials.
%Some of them are well-known see e.g.~\cite{Mi1}, and others follow readily from the known ones.
%So let $f(z)=z^2+c$ be infinitely renormalizable.

{\bf (A)}. Let $G$ be the Green function of the basin of infinity $A(\infty)=\{z| f^n(z)\to \infty, n\to\infty\}$ of $f$ with the standard normalization
at infinity $G(z)=\ln|z|+O(1/|z|)$. The external ray $R_t$ of argument $t\in {\bf S^1}=\R/\Z$ is a gradient line to the level sets of $G$
that has the (asymptotic) argument $t$ at $\infty$.
%For $t\in {\bf S^1}={\bf R}/{\bf Z}$, $R_t$ denotes the external ray of argument $t$.
$G(z)$ is called the (Green) level of $z\in A(\infty)$ and the unique $t$ such that $z\in R_t$
is called the (external) argument (or angle) of $z$.
A point $z\in J(f)$ is accessible if there is an external ray $R_t$ which lands at (i.e., converges to) $z$. Then $t$ is called an (external) argument (angle) of $z$.
%By the {\it first point of intersection} of a ray $R_t$ with a closed set $E$ we mean a point of $R_t\cap E$ with
%the minimal level.

Let $\sigma: {\bf S^1}\to {\bf S^1}$ be
the doubling map $\sigma(t)=2t(\mod 1)$. Then $f(R_t)=R_{\sigma(t)}$.

Every point $a$ of a repelling cycle $O_a$ of period $p$ is the landing point of an equal number $v$, $1\le v<\infty$, of external rays where
$v$ coincides with the number of connected components of $J(f)\setminus \{a\}$. Their arguments are permuted by $\sigma^p$ according to a rational rotation number $r/q$
(written in the lowest term); $v/q$ is the number of cycles of rays landing at $a$.
If $v\ge 2$, there is an {\bf alternative} \cite{Mi1}:

$r/q=0/1$, then $v=2$ so that each of two external ray landing at $a$ is fixed by $f^p$,

$r/q\neq 0/1$, i.e., $q\ge 2$, then $v=q$, i.e., the arguments of $q$ rays landing at $a$ form a single cycle of $\sigma^p$.

{\bf (B)}.
%Let $1=p_0<p_1<...<p_n<...$ be the sequence
%of consequtive periods of simple renormalizations of $f$, $J_n$ denotes the Julia set of the
%$n$-renormalization which contains $0$,
%$$J_\infty=\cap_{n\ge 0}\cup_{j=0}^{p_n-1}f^j(J_n)$$
%It is a compact set which contains $\omega(0)$.
All periodic points of $f$ are repelling.
Given a small Julia set $J_n$ containing $0$, sets $f^j(J_n)$, $0\le j< p_n$, are called small Julia sets of level $n$. Each $f^j(J_n)$ contains $p_{n+1}/p_n\ge 2$ small Julia sets of level $n+1$.
We have $J_n=-J_n$. Since all renormalizations are simple, for $j\neq 0$, the symmetric companion $-f^j(J_n)$ of $f^j(J_n)$ can intersect the orbit $orb(J_n)=\cup_{j=0}^{p_n-1}f^j(J_n)$ of $J_n$
only at a single point which is periodic.
On the other hand,
since only finitely many external rays converge to each periodic point of $f$, the set $J_\infty$ contains no periodic points. In particular,
each component $K$ of $J_\infty$ is wandering, i.e., $f^i(K)\cap f^j(K)=\emptyset$ for all $0\le i<j<\infty$.
All this implies that $\{x,-x\}\subset J_\infty$ if and only if $x\in K_0:=\cap_{n=1}^\infty J_n$.

{\bf Given $x\in J_\infty$, for every $n$, let $j_n(x)$ be the unique $j\in\{0,1,\cdots, p_n-1\}$ such that
$x\in f^{j_n(x)}(J_n)$. Let $J_{n,x}=f^{j_n(x)}(J_n)$ be a small Julia set of level $n$ containing $x$ and
$K_x=\cap_{n\ge 0}J_{n,x}$,
a component of $J_\infty$ containing $x$.}

In particular, $K_0=\cap_{n\ge 0}J_n$ is the component of $J_\infty$ containing $0$ and $K_c=\cap_{n=1}^\infty f(J_n)$, the component containing $c$.

Note that either $p_n-j_n(x)\to\infty$ as $n\to\infty$ or $p_n-j_n(x)=N$ for some $N\ge 0$ and all $n$,
that is, $f^N(x)\in K_0$.
This is so since the sequence of the sets $J_n$ is non-increasing, hence $J_{n,x}$ non-increasing, hence $p_n-j_n(x)$ (the time to reach $J_n$) non-decreasing.

The map $f:K_x\to K_{f(x)}$ is two-to-one if $x=0$ and one-to-one otherwise.
%$J_n=-J_n$. If $j\neq 0$, then $-f^j(J_n)$ can intersect the orbit $orb(J_n)=\cup_{j=0}^{p_n-1}f^j(J_n)$ of $J_n$
%only at a single point which is periodic.
%Indeed, otherwise there is $j_1\neq j$ such that $f^{j_1+1}(J_n)$ intersects $f^{j+1}(J_n)$ at
Moreover, for every $y\in J_\infty$, $f^{-1}(y)\cap J_\infty$ consists of two points if $y\in K_c$ and consists of a single point
otherwise. Denote
$$J_\infty'=J_\infty\setminus\cup_{j=-\infty}^\infty f^j(K_0).$$
We conclude that:

{\bf $f: J_\infty'\to J_\infty'$ is a homeomorphism}. Given $x\in J_\infty'$ and $m>0$, denote $x_m=f^m(x)$ and
$$x_{-m}=f|_{J_\infty'}^{-m}(x),$$
that is,
%moreover, for every $$x\in J_\infty\setminus \cup_{j\ge 0} f^j(K_c)$$ and every $m>0$
the only point $f^{-m}(x)\cap J_\infty$.

{\bf (C)}. Given $n\ge 0$, the map $f^{p_n}:f(J_n)\to f(J_n)$ has two fixed points:
the separating fixed point $\alpha_n$ (that is, $f(J_n)\setminus \{\alpha_n\}$ has at least two components) and
the non-separating $\beta_n$ (so that $f(J_n)\setminus \beta_n$ has a single
component).
%We denote by $C_{n,\beta}$ a periodic orbit of $f$ which contains $\beta_n$.

%\begin{remark}\label{prim}

%In other words,
%the $n$-renormalization is satellite if and only if $\beta_n=\alpha_{n-1}$ and $k_n/q_n$ is equal to the
%rotation number of the $\alpha$-fixed point of the $n-1$-renormalization.
%There is a specific such "sector" at the point $\beta_n$ as follows.
For every $n>0$, there are $0<t_n<\tilde t_n<1$ such that two rays $R_{t_n}$ and $R_{\tilde t_n}$
land at the non-separating fixed point $\beta_n\in f(J_n)$
of $f^{p_n}$ and the component $\Omega_n$ of ${\bf C}\setminus (R_{t_n}\cup R_{\tilde t_n}\cup \beta_n)$
which does not contain $0$ has two characteristic propertiers \cite{Mi1}:

(i) $\Omega_n$ contains $c$ and is disjoint with the forward orbit
of $\beta_n$,

(ii) for every $1\le j<p_n$, consider arguments (angles) of external rays which land at
$f^{j-1}(\beta_n)$.
The angles split ${\bf S^1}$ into finitely many arcs. Then the length of any such arc is bigger than
the length of the arc
$$S_{n,1}=[t_n, \tilde t_n]=\{t: R_t\subset \Omega_n\}.$$
Denote
$$t_n'=t_n+\frac{\tilde t_n-t_n}{2^{p_n}}, \ \ \tilde t_n'=\tilde t_n-\frac{\tilde t_n-t_n}{2^{p_n}}.$$
The rays $R_{t_n'}$, $R_{\tilde t_n'}$ land at a common point $\beta_n'\in f^{-p_n}(\beta_n)\cap \Omega_n$.
Introduce an (unbounded) domain $U_n$ with the boundary to be two curves
$R_{t_n}\cup R_{\tilde t_n}\cup \beta_n$ and $R_{t_n'}\cup R_{\tilde t_n'}\cup \beta_n'$.
Then $c\in U_n$ and $f^{p_n}: U_n\to \Omega_n$ is a two-to-one branched covering. Also,
$$
f(J_n)=\{z: f^{kp_n}(z)\in \overline U_n, G(f^{kp_n}(z))< 10, k=0,1,...\}.
$$
%Moreover, for any $n$, the closure of $U_{n+1}$ is contained in $U_n$.
Let
$$s_{n,1}=[t_n, t_n']\cup [\tilde t_n', \tilde t_n]$$
so that $s_{n,1}\subset S_{n,1}$ and argument of any ray to $f(J_n)$ lies in $s_{n,1}$.

Let us iterate this construction. Given $1\le j\le p_n$, let $S_{n,j}$ {\bf be one of the two arcs of ${\bf S^1}$
with end points
$$t_{n,j}=\sigma^{j-1}(t_n), \tilde t_{n,j}=\sigma^{j-1}(\tilde t_n)$$
such that
arguments of any ray to $f^{j}(J_n)$ lies in} $S_{n,j}$.
Let
$$s_{n,j}=\sigma^{j-1}(s_{n,1})=[t_{n,j}, t_{n,j}']\cup [\tilde t_{n,j}', \tilde t_{n,j}]$$
where $t_{n,j}'=\sigma^{j-1}(t_n'), \tilde t_{n,j}'=\sigma^{j-1}(\tilde t_n')$.
Then
$$s_{n,j}\subset S_{n,j}$$
and argument of any ray to $f^j(J_n)$ lies in fact in $s_{n,j}$.
Note that
\begin{equation}\label{wind}
t_{n,j}'-t_{n,j}=\tilde t_{n,j}-\tilde t_{n,j}'=\frac{\tilde t_n-t_n}{2^{p_n-j+1}}<\tilde t_n-t_n<1/2.
\end{equation}
So $\sigma^{j-1}: s_{n,1}\to s_{n,j}$ is a homeomorphism and $s_{n,j}$ has two components ('windows') $[t_{n,j}, t_{n,j}']$ and $[\tilde t_{n,j}', \tilde t_{n,j}]$ of equal length.

Let $U_{n,j}=f^{j-1}(U_n)$ and $\beta_{n,j}=f^{j-1}(\beta_n)$. The domain $U_{n,j}$ is bounded by
two rays $R_{t_{n,j}}\cup R_{\tilde t_{n,j}}$ converging to $\beta_{n,j}$ and completed by $\beta_{n,j}$ along with two rays
$R_{t_{n,j}'}\cup R_{\tilde t_{n,j}'}$ completed by their common limit point $f^{j-1}(\beta_n')$ where
$t_{n,j}'=\sigma^{j-1}(t_n'), \tilde t_{n,j}'=\sigma^{j-1}(\tilde t_n')$.

By (i)-(ii), for a fixed $n$, domains $U_{n,j}$, $1\le j\le p_n$, are pairwise disjoint.

Let $U_{n,j-p_n}$ be a component of $f^{-(p_n-j)}(U_n)$ which is contained in $U_{n,j}$.
%Let $s^1_{n,j}$ the set of arguments of rays entering $U_{n,j-p_n}$. The $s^1_{n,j}$ consists of $4$ components so that $\sigma^{p_n}$ map homeomorphically
%each of these components onto one and only one 'windows' of $s_{n,j}$.
Then
\begin{equation}\label{mapU}
f^{p_n}: U_{n,j-p_n}\to U_{n,j}
\end{equation}
is a two-to-one branched covering and
$$f^{j-1}(J_n)=\{z: f^{kp_n}(z)\in \overline U_{n, j-p_n}, G(f^{kp_n}(z))<10, k=0,1,...\}.$$
Let $s^1_{n,j}$ be the set of arguments of rays entering $U_{n,j-p_n}$. Then $s^1_{n,j}$ consists of $4$ components so that $\sigma^{p_n}$ map homeomorphically
each of these components onto one of the 'windows' of $s_{n,j}$.

Furthermore, let
$$\Omega_{n,j}=f^{j-1}(\Omega_n).$$
Unlike the map (\ref{mapU}),
the map
\begin{equation}\label{mapOmega}
f^{p_n}: U_{n,j}\to \Omega_{n,j}
\end{equation}
is a two-to-one branched covering
only {\it assuming} $f^{j-1}: \Omega_{n}\to \Omega_{n,j}$ is a homeomorphism, which holds {\it if and only if}
$\sigma^{j-1}: S_{n,1}\to \sigma^{j-1}(S_{n,1})$ is a homeomorphism. In the latter case,
$$\sigma^{j-1}(S_{n,1})=S_{n,j}.$$

{\bf Primitive} vs {\bf satellite} renormalizations.
Let $n\ge 2$ and $r_n/q_n$ be the rotation number of $\beta_n$.
%The renormalization $f^{p_n}$ is primitive
%if and only if $\beta_n$ is the landing point of exactly two rays and $k_n/q_n=0/1$;
%under the map $f^{p_{n-1}}:f(J_{n-1})\to f(J_{n-1})$ and it satellite otherwise, i.e., $q_n\ge 2$.
The next claim is well-known, we include the proof for reader's convenience.
\begin{lemma}\label{satellite}
\begin{enumerate}
\item the renormalization $f^{p_n}$ is primitive
if and only if $r_n/q_n=0/1$, the period of $\beta_n$ is $p_n$ and $\beta_n$ is the landing point of exactly two rays and they are fixed by
$f^{p_n}$,
\item points $\beta_n$, $n=1,2,\cdots$ are all different,
\item $f^{p_n}$ is satellite if and only if the $\alpha$-fixed point $\alpha_{n-1}$ of $f^{p_{n-1}}: f(J_{n-1})\to f(J_{n-1})$ coincides with
the $\beta$-fixed point $\beta_n$ of $f^{p_n}:f(J_n)\to f(J_n)$. In particular,
$\cup_{j=0}^{q_n-1} f^{jp_{n-1}}(f(J_n))\subset f(J_{n-1})$ and $p_n=q_n p_{n-1}$.
Moreover, each of $p_{n-1}$ points of the orbit of $\beta_n$ is the landing points of precisely $q_n$ rays which are permuted by $f^{p_{n-1}}$
according to the rotation number $r_n/q_n$. Completed by the landing point they split $\C$ into $q_n$ "sectors" such that the closure of each of them contains a unique "small"
Julia set of level $n$ sharing a common point with the boundary of the "sector".
\end{enumerate}
\end{lemma}
\begin{proof}
(1). $f^{p_n}$ is satellite if and only if $f(J_n)$ meets at $\beta_n$ some other iterate of $J_n$, hence, $r_n/q_n\neq 0$, and vice versa.
(2). assume $\beta:=\beta_n=\beta_m$ for some $0\le n<m$. As $p_n<p_m$, the period of $\beta_m$ is smaller than $p_n$. It follows that $f(J_n)$ contains two small Julia sets of level $m$ that meet at $\beta$, hence,
$\beta$ separates $f(J_n)$, a contradiction as $\beta_n$ does not.
(3). By (1), $f^{p_n}$ is satellite if and only if $r_n/q_n\neq 0$. Let $\tilde p_{n-1}=p_n/q_n$. Then $\tilde p_{n-1}$ is an integer and is equal to the period of $\beta_n$. It follows that $p_n$ sets $f(J_n), f^2(J_n), \cdots, f^{p_n}(J_n)$ are split into
$\tilde p_{n-1}$ connected closed subsets $E_i$, $i=1,\cdots, \tilde \tilde p_{n-1}$ where
$E_1=\cup_{j=0}^{q_n-1} f^{j \tilde p_{n-1}}(f(J_n))$ and $E_i=f^{i-1}(E_1)$, $i=1,2,\cdots, \tilde p_{n-1}$.
Moreover, $0\in E_{p_{n-1}}$ and $f(E_i)=E_{i+1}$, $i=1,\cdots, \tilde p_{n-1}-1$, $f(E_{\tilde p_{n-1}})=E_1$.
By \cite[Theorem 8.5]{mcm}, $f^{\tilde p_{n-1}}$ is a simple renormalization and $E_i$, $i=1,\cdots, \tilde p_{n-1}$ are subsets
of its $\tilde p_{n-1}$ small Julia sets. Since $1=p_0<p_1<...$ are all consecutive periods of simple renormalizations, then
$\tilde p_{n-1}=p_k$ for some $k<n$. Therefore, $\beta_n$-fixed point of $f^{p_n}:f(J_n)\to f(J_n)$ is $\alpha_k$-fixed point of
$f^{p_k}: f(J_{p_k})\to f(J_{p_k})$. As all renormalizations are simple, if $k<n-1$ that would imply that $\beta_n=\beta_{n-1}=...=\beta_{k+1}$,
a contradiction with (2).
The claim about "sectors" follows since each map $f^j$ is one-to-one in a neighborhood of $\beta_n$ and the closure of $\Omega_n$ contains
a single "small" Julia set $f(J_n)$ of level $n$ sharing a common point with $\partial\Omega_n$.
\end{proof}
%\marginpar{ChangeN 17/7}
%In particular, the renormalization $f^{p_n}$ is primitive if and only if the period of $\beta_n$ is precisely $p_n$.
%Then \cite{Mi1} $\beta_n$
%is the landing point of precisely two
%rays which are fixed by $f^{p_n}$.
%If $f^{p_n}$ is satellite, the period of $\beta_n$ is $p_{n-1}=p_n/q_n<p_n$.
%Each of $p_{n-1}$ points of the orbit of $\beta_n$ is the landing points of precisely $q_n$ rays which are permuted by $f^{p_{n-1}}$
%according to the rotation number $r_n/q_n$. Completed by the landing point they split $\C$ into $q_n$ "sectors" such that the closure of each of them contains a "small"
%Julia set of level $n$.
%Each "sector" bounded by two adjacent pair of these rays to one of and is disjoint with all other rays to the point $\beta$, and completed by $\beta_n$ contains
%exactly one of $q_n$ "small"Julia sets of the level $n$.
%For every $n$, let $e_n$ be the period of the point $\beta_n$ under the map $f^{p_{n-1}}:f(J_{n-1})\to f(J_{n-1})$.
%Then $p_n=e_nq_np_{n-1}$.

We need a more refined estimate provided the renormalization is not doubling. Assume $f^{p_n}$ is satellite so that $p_{n-1}=p_n/q_n$ with $q_n\ge 2$ and the rotation number of $\beta_n$ is $r_n/q_n\neq 0/1$.
\begin{lemma}\label{notdoubl}
Assume $f^{p_n}$ is satellite and $q_n=p_n/p_{n-1}\ge 3$, i.e., $f^{p_n}$ is not doubling. Then
\begin{equation}\label{eq:notdoubl}
\sigma^{j-1}: S_{n,1}\to\sigma^{j-1}S_{n,1} \mbox{ is a homeomorphism for } j=1,\cdots, p_{n-1}(q_n-2).
\end{equation}
In particular, given $\zeta\in (0,1/3)$, the length of $\sigma^{j-1}S_{n,1}$ tends to zero as $n\to\infty$ uniformly in $j=1,\cdots, [\zeta p_n]$ (where $[x]$ is the integer part of $x\in\R$).

Moreover, for every $1\le j\le p_{n-1}(q_n-2)$, $S_{n,j}=\sigma^{j-1}(S_{n,1})$ and
%for
%$$\Omega_{n,j}=f^{j-1}(\Omega_n),$$
the map $f^{p_n}: U_{n,j}\to \Omega_{n,j}$ is a two-to-one branched covering such that
$$f^{j}(J_n)=\{z: f^{kp_n}(z)\in \overline U_{n, j}, G(f^{kp_n}(z))<10, k=0,1,...\}.$$

%Fix $\zeta\in (0,1/3)$. Assume $f^{p_n}$ is satellite and $q_n=p_n/p_{n-1}\ge 3$, i.e., $f^{p_n}$ is not doubling. Then, for $j=1,\cdots, [\zeta p_n]$ (where $[x]$ is the integer part of $x\in\R$),
%$\sigma^{j-1}: S_{n,1}\to\sigma^{j-1}S_{n,1}$ is a homeomorphism and the length of $\sigma^{j-1}S_{n,1}$ tends to zero as $n\to\infty$ uniformly in $j$.
%Moreover, for every $1\le j\le [\zeta p_n]$, $S_{n,j}=\sigma^{j-1}(S_{n,1})$ and
%for
%$$\Omega_{n,j}=f^{j-1}(\Omega_n),$$
%the map $f^{p_n}: U_{n,j}\to \Omega_{n,j}$ is a two-to-one branched covering and
%$$f^{j}(J_n)=\{z: f^{kp_n}(z)\in \overline U_{n, j}, k=0,1,...\}.$$
\end{lemma}
\begin{proof}
Let $g=f^{p_{n-1}}: U_{n-1}\to \Omega_{n-1}$. Then $g$ is a two-to-one covering of degree $2$ and the critical value $c$.

(1) Recall that $s_{n-1,1}=[t_{n-1}, t_{n-1}']\cup [\tilde t_{n-1}', \tilde t_{n-1}]$ consists of two 'windows' so that $\sigma^{p_{n-1}}$ is orientation preserving homeomorphism of either 'window'
onto $S_{n-1,1}=[t_{n-1}, \tilde t_{n-1}]$.

(2) Consider $q_n$ rays $L_1,...,L_{q_n}$ to $\alpha_{n-1}$.
The map $g$ is a local homeomorphism near $\alpha_{n-1}$ which permutes the rays to $\alpha_{n-1}$ according to the rotation number $\nu:=r_n/q_n\neq 0, 1/2$. %See Figure \ref{fig1}.
In particular, $g$ maps any pair of adjacent rays to $\alpha_{n-1}$ onto another pair of adjacent rays to $\alpha_{n-1}$.

\begin{figure}[h]
\includegraphics[width=10cm]{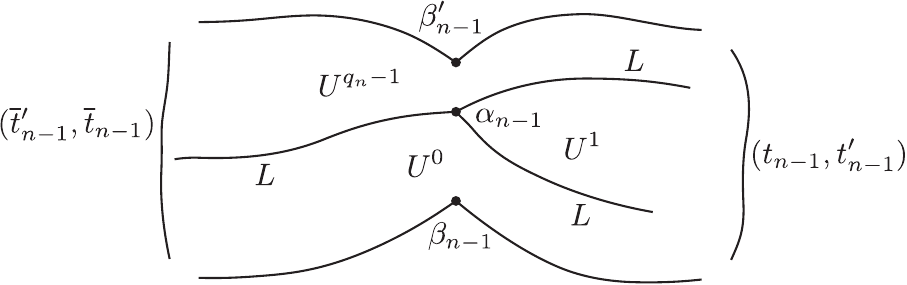}
\caption{ $q_n=3$.}
\label{fig1}
\end{figure}

(3) Not all arguments of these rays lie in a single 'window' $I$ of $s_{n-1,1}$ because otherwise, by (1), the set of those arguments would lie
in the non-escaping set of an orientation preserving homeomorphism $\sigma^{p_{n-1}}: I\to S_{n,1}$, which consists of a fixed point of this map,
a contradiction with the fact that $q_n>1$.

(4) The rays $L_j$ split $U_{n-1}$ into $q_n$ disjoint domains $U^j$, $j=0,1,...,q_n-1$. By the "ideal boundary" $\hat\partial{U^j}$ of $U^j$ we will mean the usual (topological) boundary $\partial U^j$
(in our case, the set of boundary rays completed by their landing points) along with the "boundary at infinity" which is the set of arguments of rays entering $U^j$.
Then define $\hat g$ on $\hat\partial{U^j}$ to be $g$ on $\partial U^j$ and $\sigma^{p_{n-1}}$ on the "boundary at infinity" of $U^j$.

(5) By (3), one of $U^j$, called $U^0$, has $\beta_{n-1}$ in its boundary, and another one, called $U^{q_n-1}$,
has $\beta_{n-1}'$ in the boundary. In particular, the boundary of any other $U^j$, $j\neq 0, q_n-1$, consists of a pair of adjacent rays to $\alpha_{n-1}$ whose arguments
belong to a single 'window' of $s_{n-1,1}$. Therefore, by (1), the rest of indices $j=1,..., q_n-2$ can be ordered in such a way that
$\hat g:\hat\partial{U^j}\to \hat\partial{U^{j+1}}$ is a one-to-one map for $j=1,\cdots,q_n-3$
%homeomorphism and $\sigma^{p_{n-1}}$ is a homeomorphism between
(note that the "boundary at infinity" of each $U^j$, $1\le j\le q_n-2$, consists of a single "arc at infinity"). Therefore, $g: U^j\to U^{j+1}$ is a homeomorphism for $j=1,..., q_n-3$.
The map $\hat g$ on $\hat\partial{U^{q_n-2}}$ is also a one-to-one map on its image $W=g(U^{q_n-2})$ where $W$ is bounded by two adjacent rays to $\alpha_{n-1}$.
$W$ cannot contain $U^0$ because otherwise $W$ would contain $\beta_{n-1}'$, a contradiction. Thus $W$ must contain $\beta_{n-1}'$.
%and it
%is bounded by two external rays to $\alpha_{n-1}$ which contain $\beta_{n-1}'$.
That is, $g(U^{q_n-2})$ covers $U^{q_n-1}$.

Thus, for $j=1,\cdots,q_n-3$, $g: U^j\to U^{j+1}$ is a homeomorphism,
and $g: U^{q_n-2}\to W$ is also a homeomorphism where the image $W=g(U^{q_n-2})$ covers $U^{q_n-1}$ and has two common rays with the boundary of $U^{q_n-1}$.

%As $g(\beta_{n-1}')=\beta_{n-1}$ and because of (1)-(2),
%$\hat g: \hat\partial{U^{q_n-1}}\to \hat\partial{U^0}$ is a one-to-one map too
%homeomorphism and $\sigma^{p_{n-1}}$ is a homeomorphism between
%the boundary of $U^{q_n-1}$ is mapped by $g$ in a one-to-one fashion onto the boundary of $U^0$
%while $\sigma^{p_{n-1}}$ is a homeomorphism between their "boundaries at infinity" (each consisting of two "arcs at infinity").
%where in this case the "boundary at infinity" of each $U^{q_n-1}$, $U^{0}$ consists of two "arcs at infinity".
%Therefore, $g: U^{q_n-1}\to U^0$ is a homeomorphism, too.

(6) The critical value $c$ of $g$ has a unique preimage by g (the critical point of $g$). As
%$c\in\Omega_n\subset U_{n-1}$
$c\in\Omega_n\subset \Omega_{n-1}$ and $\Omega_n$ is bounded by two adjacent rays to $\alpha_{n-1}$,
$c\in U^i$ for some $i\in\{1,\cdots,q_n-1\}$. If $i>1$, then $i-1\ge 1$ while $g$ would not be a homeomorphism of $U^{i-1}$ on its image. This shows that
%On the other hand, since $g^{q_n-1}(c)\in U^{q_n-1}$, (5) implies that
$c\in U^1=\Omega_n$.

Concluding, $U^j=g^{j-1}(\Omega_n)$, $j=1,...,q_n-2$, in particular,
$$\Omega_n, g(\Omega_n),\cdots, g^{q_n-3}(\Omega_n)\subset U_{n-1}$$
and $g^{q_n-2}:\Omega_n\to g^{q_n-2}(\Omega_n)$ is a homeomorphism,
that is, (\ref{eq:notdoubl}) holds. It implies the rest.

\iffalse

$2^{j-1}|t_n-\tilde t_n|$ is less than $1$ and tends to zero as $n\to\infty$ uniformly in $j$.
By a well-known relation \cite{???},
\begin{equation}\label{digits}
|t_n-\tilde t_n|=\frac{(\tilde d_{n-1}-d_{n-1})(2^{p_{n-1}}-1)}{2^{p_n}-1}.
\end{equation}
where $0<d_{n-1}<\tilde d_{n-1}<2^{p_{n-1}}$ are two integer "digits" determined by:
$t_{n-1}=d_{n-1}/(2^{p_{n-1}}-1)$ and $\tilde t_{n-1}=\tilde d_{n-1}/(2^{p_{n-1}}-1)$.
Hence,
$$2^{j-1}|t_n-\tilde t_n|\le \frac{2^{\zeta p_{n-1}q_{n}}(\tilde d_{n-1}-d_{n-1})(2^{p_{n-1}}-1)}{2^{q_{n}p_{n-1}-1}}\le
\frac{2^{\zeta p_{n-1}q_{n}}(2^{p_{n-1}}-2)(2^{p_{n-1}}-1)}{2^{p_{n-1}q_n}-1}$$
$$<\frac{2^{\zeta p_{n-1}q_{n}}2^{2p_{n-1}}}{2^{p_{n-1}q_n}}
=\frac{1}{2^{p_{n-1}((1-\zeta)q_n-2)}}\le 2^{-p_{n-1}(3(1-\zeta)-2)}.$$
The claim follows.

\fi

\end{proof}
{\bf (D)}. Given a compact set $Y\subset J(f)$ denote by $(\tilde Y)_f$ (or simply $\tilde Y$, if the map is fixed)
the set of arguments of the external rays which have their limit sets contained in $Y$.
It follows from (C) that $\tilde K_c=\bigcap_{n=1}^\infty\{ [t_n, t_n']\cup
[\tilde t_n', \tilde t_n]\}$, i.e., it is either a single-point set or
a two-point set.

Since $\tilde K_c$ contains at most two angles, $K_c$ contains at most two different accessible points.
More generally, given $x\in J'_\infty$
%Let $x\in \cup_{\epsilon>0} E_\epsilon\subset E\subset \tilde E$, where
%$\tilde E=\supp (\mu)\setminus\cup_{j=-\infty}^\infty f^j(K_0)$. Recall that the ray $R_{t(x)}$ lands at $x$.
%For every $n$, there is a unique $1\le j_n(x)\le p_n$ such that $x\in f^{j_n(x)}(J_n)$.
%Observe that $p_n-j_n(x)\to\infty$ as $n\to\infty$ if and only if there is no $N\ge 0$ such that $f^N(x)\in K_0$.
let
$$s_{n,j_n(x)}=[t_{n,j_n(x)}, t_{n,j_n(x)}']\cup [\tilde t_{n,j_n(x)}', \tilde t_{n,j_n(x)}].$$
Then $s_{n+1,j_{n+1}(x)}\subset s_{n,j_{n}(x)}$ so that
$$s_{\infty, x}:=\cap_{n>0}s_{n,j_n(x)}$$
is not empty and consists of either one or two components.
Since $p_n-j_n(x)\to\infty$ for $x\in J'_\infty$ we conclude using (\ref{wind}):

$s_{\infty, x}$ {\it consists of either a single point or two different points.
In particular, for
any component $K$ of $J_\infty$ which is not one of $f^{-j}(K_0)$, $j\ge 0$, there is either one or
two rays tending to $K$.}

\

{\bf From now on, $\mu$ is an $f$-invariant probability ergodic measures supported in $J_\infty$:
$\supp \mu\subset J_\infty$, and having a positive
Lyapunov exponent}
$$\chi(\mu):=\int\log|f'|d\mu>0.$$
{\bf (E)}.
%The map $f:K_x\to K_{f(x)}$ is two-to-one if $x=0$ and one-to-one otherwise.
%It follows that, for every $x\in J_\infty\setminus \cup_{j\ge 0} f^j(K_c)$ and every $m>0$ there is one and only one point $x_{-m}\in f^{-m}(x)\cap J_\infty$.
%By \cite{...}, $\mu$-almost every point is accessible by an external ray.
%In the course of the proof of Theorem \ref{thm:meas+} we construct domains which are disks cut by external rays (and on which
%an inverse branch of $f^{p_n}$ is well-defined such that it maps the intersection of a level $n$ Julia set with this domain into itself).
We start with the following basic statement. Parts (i)-(ii) are easy consequences of the invariance of $\mu$ and (B)
while (iii) is a part of Pesin's theory as in \cite[Theorem 11.2.3]{PU} coupled with the structure of $f:J_\infty\to J_\infty$, see (B).
Recall that $J_\infty'=J_\infty\setminus \cup_{j=-\infty}^{\infty} f^j(K_0)$.
\begin{prop}\label{lyplus}
%Let $\mu$ be an $f$-invariant probability ergodic measure
%with a positive Lyapunov exponent $\chi(\mu)>0$ such that $\supp \mu\subset J_\infty$.
%The following holds.

(i) For every $n$ and $0\le j<p_n$, $\mu(f^j(J_n))=1/p_n$.

(ii) $\mu$ has no atoms and $\mu(K)=0$ for every component $K$ of $J_\infty$.

(iii) $\mu(J_\infty')=1$ and $f: J_\infty'\to J_\infty'$ is a $\mu$-measure preserving homeomorphism.
%moreover, there is a subset $E$ of $\tilde E$ such that $\mu(E)=1$, $f: E\to E$ is onto and one-to-one
There exists a measurable positive function $\tilde r(x)>0$ on $J_\infty'$ such that for $\mu$-almost every $x\in J_\infty'$,
and all $n\in {\bf N}$, if $x_{-n}$ is the unique point of $J_\infty'$ with $f^n(x_{-n})=x$,
then a (univalent) branch $g_n: B(x, \tilde r(x))\to {\bf C}$ of $f^{-n}$ is well-defined such that
$g_n(x)=x_{-n}$,
\end{prop}
\begin{remark} The branch $g_n$ of $f^{-n}$ depends on $n$ and $x_{-n}$ but
it should be clear from the context which points $x$ and $x_{-n}$ are meant.
\end{remark}
%Given $n\ge 0$, the map $f^{p_n}:f(J_n)\to f(J_n)$ has two fixed points:
%a separating fixed point $\alpha_n$ (so that $f(J_n)\setminus \alpha_n$ has at least two components) and
%the non-separating $\beta_n$ (so that $f(J_n)\setminus \beta_n$ has a single
%component).
%We denote by $C_{n,\beta}$ a
%periodic orbit of $f$
%which contains $\beta_n$. Recall that the renormalization $f^{p_n}:f(J_n)\to f(J_n)$ is called
%{\it primitive} if the rotation number of the point $\beta_n$ is equal to $0$ and {\it satellite}
%otherwise. Note that it is primitive if and only if $\beta_n$ is the landing point of precisely two
%rays which are fixed by $f^{p_n}$. In this (and only this) case the period of $\beta_n$ is precisely $p_n$.

Using the Birkhoff Ergodic Theorem and Egorov's theorem, Proposition \ref{lyplus} implies immediately (e1)-(e3) of the next corollary. The proof of (e4)-(e5) is given right after it.
\begin{coro}\label{eee}
For every $\epsilon>0$, there exists a closed set $E'_{\epsilon/2}\subset J_\infty'$
and constants $\rho=\rho(\epsilon)>0$, $\kappa=\kappa(\epsilon)\in (0,1)$ such that:

($e_1$) $\mu(E'_{\epsilon/2})>1-\frac{\epsilon}{2}$,

($e_2$) there exists another closed set $\hat E_{\epsilon/2}$ such that $E'_{\epsilon/2}\subset\hat E_{\epsilon/2}\subset J'_\infty$
as follows. For every $x\in \hat E_{\epsilon/2}$ and every $m>0$
there exists a (univalent) branch
$g_m: B(x, 3\rho)\to {\bf C}$ of $f^{-m}$ such that $g_m(x)=x_{-m}$
and $|g'_m(x_1)/g'_m(x_2)|<2$, for every $x_1, x_2\in B(x,2\rho)$. Moreover,
$m^{-1}\ln|g'_m(x)|\to -\chi(\mu)$ as $m\to \infty$ uniformly in $x\in  E'_{\epsilon/2}$,
%in particular, $m^{-1}\ln|Dg_m(x)|<-\chi(\mu)/2$ for all $m$ large enough,

($e_3$) for every $x\in E'_{\epsilon/2}$ there exists a sequence of positive integers
$n_j=n_j(x)$, $j=1,2,...$, such that $j/n_j\ge \kappa$ and $f^{n_j}(x)\in \hat E_{\epsilon/2}$ for all $j$,
%(in fact, $\{n_j\}_{j=1}^\infty=\{n\in\N: f^n(x)\in\tilde E_{\epsilon/2}\}$),

($e_4$) given $x\in J_\infty$ and $n\ge 0$, let $j_n(x)$ be the unique $1\le j< p_n$
%Let $x\in \cup_{\epsilon>0} E_\epsilon\subset E\subset \tilde E$, where
%$\tilde E=\supp (\mu)\setminus\cup_{j=-\infty}^\infty f^j(K_0)$. Recall that the ray $R_{t(x)}$ lands at $x$.
such that $x\in f^j(J_n)$.
Then $p_n-j_n(x)\to \infty$ as $n\to \infty$ uniformly in $x\in E'_{\epsilon/2}$,

($e_5$) for $s_{n,j_n(x)}=[t_{n,j_n(x)}, t_{n,j_n(x)}']\cup [\tilde t_{n,j_n(x)}', \tilde t_{n,j_n(x)}]$,
we have: $s_{n+1,j_{n+1}(x)}\subset s_{n,j_{n}(x)}$ and
$$|t_{n,j_n(x)}-t_{n,j_n(x)}'|=|\tilde t_{n,j_n(x)}'-\tilde t_{n,j_n(x)}|\to 0$$
as $n\to \infty$ uniformly in $x\in  E'_{\epsilon/2}$.
%In particular,
%$$s_{\infty, x}:=\cap_{n>0}s_{n,j_n(x)}$$
%is not empty and consists of either a single point or two different points.}
\end{coro}
Proof of ($e_4$)-($e_5$): assuming the contrary in ($e_4$), we find some $N$ and sequences $(n_k)\subset\N$ and $(x_k)$, $x_k\in E'_{\epsilon/2}$,
such that $p_{n_k}-j_{n_k}(x_k)=N$ (see (B)) hence, $x_k\in f^{-N}(J_{n_k})$, for all $k$. Since $E_{\epsilon/2}$ is closed, one can assume $x_k\to x\in E'_{\epsilon/2}\subset J'_\infty$.
Hence, $x\in f^{-N}(K_0)$, a contradiction. Now, for ($e_5$)  using ($e_4$),
%\begin{equation}\label{uns}
$t_{n,j_n(x)}' - t_{n,j_n(x)} = \tilde t_{n,j_n(x)} -  \tilde t_{n,j_n(x)}'<\frac{1}{2^{p_{n}-j_n(x)}}\to 0$
%\end{equation}
uniformly in $x$.

\section{External rays to typical points}\label{access}
We define a {\it telescope} following essentially~\cite{Pr}.
Given $x\in J(f)$, $r>0$, $\delta>0$, $k\in {\bf N}$ and $\kappa\in (0,1)$, an
$(r, \kappa, \delta, k)$-telescope at $x\in J$ is collections of times
$0=n_0<n_1<...<n_k$ and disks $B_l=B(f^{n_l}(x), r)$, $l=0,1,...,k$ such that,
for every $l>0$:
(i) $l/n_l>\kappa$, (ii) there is a univalent branch $g_{n_l}: B(f^{n_l}(x), 2r)\to {\bf C}$ of $f^{-n_l}$
so that $g_{n_l}(f^{n_l}(x))=x$ and, for $l=1,...,k$, $\dist(f^{n_{l-1}}\circ g_{n_l}(B_l), \partial B_{l-1})>\delta$
(clearly, here $f^{n_{l-1}}\circ g_{n_l}$ is a branch of $f^{-(n_l-n_{l-1})}$ that maps $f^{n_l}(x)$ to
$f^{n_{l-1}}(x)$).
%$D_t$ be a component of the intersection of
%the basin of infinity with $B(f^{n_t}(x), r)$.
%Here $D_k$ is chosen by us, and other $D_t$, $t=0,1,...,k-1$, are determined accordingly ??.
The trace of the telescope is a collection of sets $B_{l,0}=g_{n_l}(B_l)$, $l=0,1,...,k$.
We have: $B_{k,0}\subset B_{k-1,0}\subset...\subset B_{1,0}\subset B_{0,0}=B_0=B(x,r)$.

By the {\it first point of intersection} of a ray $R_t$, or an arc of $R_t$, with a set $E$ we mean a point of $R_t\cap E$ with
the minimal level (if it exists).
\begin{theorem}\label{tele}\cite{Pr}
Given $r>0$, $\kappa\in (0,1)$, $\delta>0$ and $C>0$ there exist
$M>0$,
$\tilde l,\tilde k\in \N$ and $K>1$
such that for every $(r, \kappa, \delta, k)$-telescope the following hold.
Let $k>\tilde k$. Let $u_0=u$ be any point at the boundary of $B_k$ such that $G(u)\ge C$.
Then there are indexes $1\le l_1<l_2<...<l_j=k$ such that $l_1<\tilde l$,
$l_{i+1}<K l_i$, $i=1,...,j-1$ as follows.
Let $u_k=g_{n_k}(u)\in \partial B_{k,0}$ and
let $\gamma_k$ be an infinite arc of
an external ray through $u_k$
between the point $u_k$ and $\infty$.
%and let $\gamma_k$ be an extension of the arc $g_{n_k}(\gamma)=f^{-n_k}(\gamma)$
%of another external ray, which joints a preimage $u_k=g_k(u)\in \partial B_{k,0}$ with infinity.
Let $u_{k,k}=u_k$ and, for $l=1,...,k-1$,
let $u_{k,l}$ be the first point of intersection of
$\gamma_k$ with $\partial B_{l,0}$.
Then, for $i=1,...,j$,
$$G(u_{k,l_i})>M 2^{-n_{l_i}}.$$
\end{theorem}
%{\bf (F)}.
%Proposition~\ref{lyplus} and Theorem~\ref{thm:meas+}(1) are easy corollary of (A)-(F). Here is another one:
Next corollary of Theorem \ref{tele} is a key one.
\begin{prop}\label{acc}
Given $\epsilon>0$ there exists a closed set $E_\epsilon$ as follows. First, $\mu(E_\epsilon)>1-\epsilon$ and
$E_\epsilon\subset E'_{\epsilon/2}$ where $E'_{\epsilon/2}$ is the set defined
in (E) and satisfies ($e_1$)-($e_5$).
%In particular, ($e_2$)-($e_5$) hold for $x\in E_\epsilon$.
There exists $r=r(\epsilon)>0$ and, for each $\nu>0$ there is $C(\nu)>0$ as follows.

(1) Let $x\in E_\epsilon$. Then $x$ is the landing point of an external ray $R_{t(x)}$ of argument $t(x)$.
Moreover, the first intersection of $R_{t(x)}$ with $\partial B(x,\nu)$ has the level at least $C(\nu)$.

(2) for each $n$, a branch $g_n: B(x,2r)\to {\bf C}$ of $f^{-n}$ is well-defined such that
$g_n(x)=x_{-n}$,
$|g'_n(x_1)/g'_n(x_2)|<2$, for every $x_1, x_2\in B(x,r)$ and
$n^{-1}\ln|g'_n(x)|\to -\chi(\mu)$ as $m\to \infty$ uniformly in $x\in  E_\epsilon$,
%in particular, $n^{-1}\ln|Dg_n(x)|<-\chi(\mu)/2$ for all $n$ large enough

(3) if $x'=g_n(x)\in E_\epsilon$, then $f^n(R_{t(x')})=R_{t(x)}$.
%(4) given $\sigma>0$ there is $\delta>0$ such that $|x_1-x_2|<\sigma$ for some $x_1,x_2\in E_\epsilon$ whenever $|t(x_1)-t(x_2)|<\delta$.
\end{prop}

\begin{proof} (1)-(2) will hold already for the set $E'_{\epsilon/2}$ which follows from Theorem \ref{tele} as in \cite{Pr} and uses only that $\mu$ has a positive exponent;
(3) will follow in our case as we shrink a bit the set $E'_{\epsilon/2}$ since each point $x\in J_\infty'$ admits at most two external arguments.

Here are details. Let $r=\rho(\epsilon)$ and $\kappa=\kappa(\epsilon)$ as in the properties ($e_2$)-($e_3$) of the set $E'_{\epsilon/2}$.
Then, by ($e_2$)-($e_3$), there is $\delta>0$ such that, for each $k$, every $x\in E'_{\epsilon/2}$ admits
$(r, \kappa, \delta, k)$-telescope with the times $0=n_0<n_1<n_2<...<n_k$
that appear in the property ($e_3$) of $E'_{\epsilon/2}$.
On the other hand, there exists $L_r>0$ such that
for every $z\in J(f)$ there is a point $u(z)\in\partial B(z,r)$ with the level $G(u(z))>L_r$.
This is so due to $\bigcup_{L>0}\{G(z)\ge L\}= A_\infty$.

\noindent Given this $C=L_r$, let $M$, $\tilde l$ and $\tilde k$ be as in Theorem \ref{tele}.

Let $x\in E'_{\epsilon/2}$ and $n_1<n_2<...<n_k<...$ as in ($e_3$). Fix $k>\tilde k$.
Let $B_{k,0}(x)\subset B_{k-1,0}(x)\subset\cdots\subset B_{1,0}(x)\subset B_{0,0}(x)$ be the corresponding trace.
%Define $C_0=\inf_{y\in J(f)} \max_{z\in B(y, r)} G(z)$. It is easy to see that $C_0>0$.
%The constant $C_0$ and the parameters $r$, $\delta$,
%$\tilde n\in {\bf N}$,
%$\kappa\in (0,1)$ as well as segments of integers $(n_i)_{i=1}^k$ which are introduced above are independent on $x\in X$.
%For each $x\in X$ we choose a point $u(x)\in \partial B(x,r)$ such that $G(u(x))\ge C_0$.
%Let $\gamma(x)$ be an arc of an external ray through $u(x)$ which joins $x$ to $\infty$.
By Theorem~\ref{tele}, there are
%$\tilde l$ and $\tilde k$ such that for each $k>\tilde k$ and each $x\in X$ the following hold. There are
$1\le l_{1,k}(x)<l_{2,k}(x)<\cdots<l_{j^x_k,k}(x)=k$ such that
$l_{1,k}(x)<\tilde l$, $l_{i+1,k}(x)<K l_{i,k}(x)$, $i=1,\cdots,j^x_k-1$.
%(in particular, $l_{i,k}(x)<2^{i}\tilde l$)
Let $\gamma_k(x)$ be an arc of an external ray between the point $u_k(x):=g_{n_k}(u(f^{n_k}(x))$ and $\infty$.
Let $u_{k,l}(x)$ be the first intersection of $\gamma_k(x)$ with
$\partial B_{l,0}(x)$. Then, for $i=1,\cdots,j^x_k-1$,
\begin{equation}\label{level}
G(u_{k,l_{i,k}}(x))>M 2^{-n_{l_{i,k}(x)}}>M 2^{-l_{i,k}(x)/\kappa}.
\end{equation}
For all $i=1,\cdots,j^x_k-1$,
\begin{equation}\label{acc1}
i\le l_{i,k}(x)<K^i \tilde l.
\end{equation}
%Recall that $B_{l,0}(x)=g_{l,x}(B(f^{lj_0}(x),r)$, hence, there is $\Lambda>1$ so that $B(x,\Lambda^{-l})\subset B_{l,0}(x)\subset B(x,\lambda^{-l})$ for all $l$ and $x$.
Denote by $t_k(x)$ the argument of an external ray that contains the arc $\gamma_k(x)$.

Now, given a sequence
\begin{equation}\label{kj}
k_1<k_2<...<k_m<...
\end{equation}
such that $k_1>\tilde k$,
we get a sequence of arguments $t_{k_m}(x)$ and a sequence
of arcs $\gamma_{k_m}(x)$ of external rays of the corresponding arguments $t_{k_m}(x)$.
Passing to a subsequence in the sequence $(k_m)$, if necessary, one can assume that $t_{k_m}(x)\to \tilde t(x)$, for some argument $\tilde t(x)$.

Fix any $\nu\in (0,r)$ and choose $\tilde k_0>\tilde k$ such that,
$$2\exp(-K^{\tilde k_0-2}\tilde l \chi(\mu))<\nu \mbox { and let } C(\nu)=M (2^{-1/\kappa})^{\tilde l K^{\tilde k_0}}.$$
%Fix $l_0$ so that $\lambda^{-l_0} r<\mu/2$ and let
%\begin{equation}\label{acc2}
%C(\mu)=\frac{M}{2^{K^{l_0} \tilde l j_0}}.
%\end{equation}
%There is $m_0$ such that for all $m>m_0$ and all $l>l_0$,
%$B_{l_0,0}(x_m)\subset B(y,\mu)$.
Then, by Theorem \ref{tele}, for each $k_m>k_0$, the first intersection of the ray $R_{t_{k_m}}(x)$
with the boundary of $B(x,\nu)$ has the level at least $C(\nu)$. It follows, for any $0<C<C(\nu)$, the sequence of arcs
of the rays $R_{t_{k_m}(x)}$ between the levels $C$ and $C(\nu)$ do not exit $B(x,\nu)$ for all $k_m>k_0$.
As $t_{k_m}(x)\to \tilde t(x)$, it follows that the arc of the ray $R_{\tilde t(x)}$ between levels $C$ and $C(\nu)$ stays in
$B(x,\nu)$ too.
As $\nu>0$
and $C\in (0, C(\nu))$ can be chosen arbitrary small, $R_{\tilde t(x)}$ must land at $x$ and satisfy (1) with $t(x)$ replaced by $\tilde t(x)$.

Let us call the above procedure of getting $\tilde t(x)$ from the constants $r$, $L_r$, the point $x\in E'_{\epsilon/2}$ and the sequence (\ref{kj}) the $(r, L_r, x, (k_m))$-{\it procedure}.

Note that (2) is property ($e_2$) of the set $E'_{\epsilon/2}$.

{\it In order to satisfy property (3), we shrink the set $E'_{\epsilon/2}$ and correct $\tilde t(x)$ changing it to some $t(x)$ (if necessary)} as follows.
Using the Birkhoff Ergodic Theorem and Egorov's theorem, choose a closed subset $E_\epsilon$ of $E'_{\epsilon/2}$ such that $\mu(E_\epsilon)>1-\epsilon$
and, for each $x\in E_\epsilon$, the set $\mathcal{N}(x):=\{N\in\N: f^N(x)\in E'_{\epsilon/2}\}$ is infinite.
Note that $\mathcal{N}(x)\subset\{n_k\}_{k=1}^\infty$.
We have proved that, for each $N\in\mathcal{N}(x)$, (1) holds for the point $f^N(x)$ instead of $x$,
in particular, $\tilde t(f^N(x))$ is an argument of $f^N(x)$.
On the other hand, by (D1), each $y\in E_\epsilon$ admits at most two external arguments, hence,
all possible external arguments of the forward orbit $f^n(x)$, $n\ge 0$, belong to at most
two different orbits of $\sigma: S^1\to S^1$. Hence,
there is one of those orbits, $O=\{\sigma^n(t(x))\}_{n\ge 0}$ for some $t(x)$, such that the intersection
$O\cap\{\tilde t(f^N(x)): N\in\mathcal{N}(x)\}$ is an infinite set, so that
$\tilde t(f^{n_{k_m(x)}}(x))=\sigma^{n_{k_m(x)}}(t(x))$ for an infinite sequence $(k_m(x))_{m\ge 1}$.

Let's start over with the $(r/2, C(r/2), x, (k_m(x)))$-procedure for the point $x$ and the sequence $\{k_j(x)\}$.
Then, by the construction, $t_{k_m(x)}=t(x)$ for all $m$, hence, (1) holds with $t(x)$ instead of the previous $\tilde t(x)$.
If $y\in E_\epsilon$ is any other point of the grand orbit $\{f^n(x): n\in\Z\}$ (remember that $f:J'_\infty\to J'_\infty$ is invertible), the $(r/2, C(r/2), y, (k_m))$-procedure works for $y$
with the same (perhaps, truncated) sequence $k_1(x)<k_2(x)<...$, which ensures that
(3) holds (for the corrected arguments) too.
\end{proof}
\begin{remark} Given $t(x)$, we cannot just set $t(f^n(x))=\sigma^n (t(x))$ to satisfy property (3)
because this would change $\kappa$ in the definition of telescope, so we might loose property (1).
%???????????? because of lack of (Fi)  (crucial for small $\ell$'s)?????????????????
Notice that correcting (flipping) $\tilde t(x)$ to $t(x)$ does not change $C(\nu)$
%??????????????which depend on $x$ but not on external rays landing at it.?????????????
The same for flipping any $t(y)$ in the grand orbit of $x$.
But the flipping can make $f^\ell(R_{t(y)})=R_{t(f^N(x)}$ for $f^\ell(y)=f^N(x)$ where $N=n_{k_m}$ with
$G(R_{t(f^\ell(y))}\cap \partial B(f^\ell(y),r/2)>L_{r/2}$, thus yielding (3).
\end{remark}

\section{Lemmas}
Recall that for any $z\in J'_\infty$ we define $z_m=f^m(z)$ for any $m\in\Z$. This makes sense since $f$ is invertible on $J'_\infty$, see (E).
\begin{lemma}\label{p0}
Let $z(k)\in \cup_{j=0}^{p_{n_k}-1}f^{j}(J_{n_k})$ where $n_k\nearrow\infty$.

{\bf (a)} If $z(k)\to z$ then $z\in J_\infty$.

{\bf (b)} $z\in J_{n,x}\cap J'_\infty$ yields $z_{\pm p_n}\in J_{n,x}$.
If, additionally to (a), $z(k)\in J'_\infty$ for all $k$
%(where the set $J'_\infty$ is as Proposition \ref{lyplus})
and $w(k)\to w$ where $w(k)=z(k)_{e p_{n_k}}$, where $e$ is always either $1$ or $-1$ then $z$ and $w$ are in the same component of $J_\infty$.

{\bf (c)}
%Given $\sigma>0$ there is $\delta(\sigma)>0$ such that $|x_1-x_2|<\sigma$ for some $x_1,x_2\in E_\epsilon$ whenever $|t(x_1)-t(x_2)|<\delta(\sigma)$.
If $z(k)\in E_\epsilon$ for all $k$ and $t(z(k))\to t$ (where $E_\epsilon$, $t(z(k))$ are defined in Proposition \ref{acc}), then the ray
$R_t$ lands at the limit point $z$. In particular, given $\sigma>0$ there is $\Delta(\sigma)>0$ such that $|x_1-x_2|<\sigma$ for some $x_1,x_2\in E_\epsilon$ whenever $|t(x_1)-t(x_2)|<\Delta(\sigma)$.
\end{lemma}
\begin{proof}
(a) Assume the contrary. Then there is $n$ such that $d:=\dist(z, \cup_{j=0}^{p_{n}-1}f^j(J_{n}))>0$.
As, for any $n_k\ge n$, $z(k)\in \cup_{j=0}^{p_{n_k}-1}f^j(J_{n_k})$ where the latter union is a subset of $\cup_{j=0}^{p_{n}-1}f^j(J_{n})$,
the distance between $z$ and $z_k$ is at least $d$, a contradiction.

(b) $z_{\pm p_n}\in J_{n,x}$ by combinatorics and definitions of points $z_{m}$. In particular, for every $k$, $z(k)$ and $w(k)$ are in the same component $f^{j_k}(J_{n_k})$
of $\cup_{j=0}^{p_{n_k}-1}f^{j}(J_{n_k})$.
By (a), any limit set $A$ of the sequence of compacts $(f^{j_k}(J_{n_k}))$ in the Hausdorff metric is a subset of $J_\infty$.
On the other hand, $A$ is connected as each set $f^{j_k}(J_{n_k})$ is connected. This proves (b).

(c) We prove only the first claim as the second one directly follows from it. Fix any $\nu\in (0,r)$ and choose $k_0$ such that for any $k>k_0$,
$B(z(k), \nu)\subset B(z,11/10\nu)$.
%Let $C(\rho)=M (2^{-1/\kappa})^{\tilde l 2^{\tilde k_0}}$.
%Fix $l_0$ so that $\lambda^{-l_0} r<\mu/2$ and let
%\begin{equation}\label{acc2}
%C(\mu)=\frac{M}{2^{K^{l_0} \tilde l j_0}}.
%\end{equation}
%There is $m_0$ such that for all $m>m_0$ and all $l>l_0$,
%$B_{l_0,0}(x_m)\subset B(y,\mu)$.
Then, by Proposition \ref{acc}, part (1), for each $k>k_0$, the first intersection of the ray $R_{t(z_k)}$
with the boundary of $B(z,(11/10) \nu)$ has the level at least $C(\nu)$.
 %%% $\tilde C(\nu):=C(11/10\nu)$.
It follows, for any $0<C<C(\nu)$, the sequence of arcs
of the rays $R_{t_{z_k}}$ between the levels $C$ and $ C(\nu)$ do not exit $B(z,(11/10) \nu)$ for all $k>k_0$.
As $\nu>0$
and $C\in (0, \tilde C(\nu))$ can be chosen arbitrary small, $R_t$ must land at $z$.
%The proof of (c) repeats the proof of Proposition \ref{acc}, part (1).
%In fact, this follows from it.
%Namely,
%using notations of Proposition \ref{acc},
%fix any $\nu\in (0,r)$ and choose $k_0$ such that for any $k>k_0$,
%$B(z_k, \nu)\subset B(z,11/10\nu)$.
%Let $C(\rho)=M (2^{-1/\kappa})^{\tilde l 2^{\tilde k_0}}$.
%Fix $l_0$ so that $\lambda^{-l_0} r<\mu/2$ and let
%\begin{equation}\label{acc2}
%C(\mu)=\frac{M}{2^{K^{l_0} \tilde l j_0}}.
%\end{equation}
%There is $m_0$ such that for all $m>m_0$ and all $l>l_0$,
%$B_{l_0,0}(x_m)\subset B(y,\mu)$.
%Then, by Proposition \ref{acc}, for each $k>k_0$, the first intersection of the ray $R_{t(z_k)}$
%with the boundary of $B(z,\nu)$ has the level at least $\tilde C(\nu):=C(11/10\nu)$. It follows, for any $0<C<\tilde C(\nu)$, the sequence of arcs
%of the rays $R_{t_{z_k}}$ between the levels $C$ and $\tilde C(\nu)$ do not exit $B(z,\nu)$ for all $k>k_0$. As $\nu>0$
%and $C\in (0, \tilde C(\nu))$ can be chosen arbitrary small, $R_t$ must land at $z$.
\end{proof}
%{\bf (G)}.
%Let us introduce another notations.

%intersection $u_2$ with $R_{\hat t_2}$ and then back along $R_{\hat t_2}$ from $u_2$ to $x_2$ belongs to
%$\Omega_n$ and $B(x_1,\hat r)$ simultaneously.
%\end{proof}

By lemma \ref{p0}(c), if arguments $t(x),t(x')$ of $x,x'\in E_\epsilon$ are close then $x,x'$ are close as well.

\begin{definition}\label{delta} Given $\epsilon$ and $\rho$ we define $\delta$ as follows.
First, for $\hat r\in (0,1)$ and $\hat C>0$, we define
$\hat\delta=\hat\delta(\hat r,\hat C)>0$.
%and $\tilde n\in {\bf N}$
Namely, let $C_0>0$ be so that the distance between the equipotential of level $C_0$ and $J(f)$ is bigger than $1$. Then $\hat\delta=\hat\delta(\hat r/2,\hat C)>0$ is such that for any $C\in [\hat C, C_0]$,
%and any two points $w_1,w_2\in A(\infty)$,
if $w_1,w_2$ lie on the same equipotential $\Gamma$ of level $C$ and the difference between external arguments of $w_1,w_2$ is less than
$\hat\delta$ then the length of the shortest arc of the equipotential $\Gamma$ between $w_1$ and $w_2$ is less than $\hat r/2$.
Apply Lemma \ref{p0}(c)
with $\sigma=\rho/4$ and find the corresponding $\Delta(\rho/4)$.
%Let $\tilde C=C(\rho/2)$ where $C(\nu)$ is defined in Proposition \ref{acc}.
%Apply Claim $\tilde 1$ with $\hat r=\rho$ and $\hat C=\tilde C$ and find the corresponding $\hat\delta$.
Let
$$\delta=\delta(\epsilon, \rho):=\min\{\hat\delta(\rho, C(\rho/2)), \Delta(\frac{\rho}{4})\}$$
where $C(\nu)$ is defined in Proposition \ref{acc}.
\end{definition}

In the next two lemmas we construct curves with special properties. The idea is as follows.
Let $x\in E_\epsilon\cap J_{n,x}$. Then $x_{-p_n}\in J_{n,x}$. It is easy to get a curve $\gamma$ in $A(\infty)$ as follows:
begin with an arc from a point $b\in R_{t(x)}$
to $g_{p_n}(b)$ and then iterate this arc by $g_{p_n}$. In this way we get a curve $\gamma$ such that $g^{p_n}(\gamma)\subset \gamma$, hence,  $\gamma$ lands at a fixed point $a$ of $f^{p_n}$.
We show in the next lemma (in a more general setting) that if both points $x,x_{-p_n}$ are either in the range of the covering (\ref{mapU}) (condition (I)) or in the range of the covering (\ref{mapOmega})
(condition (II)) then
$a\in J_{n,x}$. This implies that $a$ has to be the $\beta$-fixed point of $f^{p_n}: J_{n,x}\to J_{n,x}$. In Lemma \ref{p1} assuming additionally that $f^{p_n}$ is satellite, we 'rotate' the curve $\gamma$
by $g_{p_{n-1}}$ to put the set $J_{n,x}$ in a 'sector' bounded by $\gamma$ and by its 'rotation'. In Lemma \ref{p3}-\ref{doublem} we consider the case of doubling for which the condition (II) usually does not hold.

\smallskip

In what follows, we use the following notation: given $p,q\in\N$ , let
$$E_{\epsilon,p,q}=\cap_{j=0}^{q-1} f^{jp}(E_\epsilon).$$
It is a closed subset of $E_\epsilon$ of points $x$ such that $x_{-jp}\in E_\epsilon$ for $j=0,1,\cdots,q-1$.
As $f: J'_\infty\to J'_\infty$ is a $\mu$-automorphism, $\mu(E_{\epsilon,p,q})>1-q\epsilon$. Notice that this bound is independent of $p$.

\smallskip

For every $n>0$
consider the closed set $E_{\epsilon, p_n, q}$.
%Under the conditions of Lemma~\ref{key}, there is $\delta>0$ as follows.
Let $x\in E_{\epsilon, p_n, q}$. Denote for brevity
$$x^k:=x_{-kp_n}  \mbox{ and }  R^k:=R_{t(x^k)}, \ \ k=0,1,...,q-1.$$
By Lemma \ref{p0}(b), $x^k\in J_{n,x}$. Hence,
$t(x^k)\in s_{n, j_n(x)}\subset S_{n,j_n(x)}$, $0\le k\le q-1$.

Recall that for a semi-open curve $l: [0,1)\to \C$, we say that $l$ lands at, or tends to, or converges to a point $z\in\C$ if there exists $\lim_{t\to1}l(t)=z$. Then $l(0)$, $z$ are endpoints of the curve and $l(0)$ is called also the starting point of $l$.

\begin{lemma}\label{fixed}
Fix $\epsilon>0$ and consider the set $E_\epsilon$ with the corresponding constant
$r(\epsilon)>0$.
%$\kappa(\epsilon)$,$\tilde l(\epsilon)$ and $M(\epsilon)>0$.
Fix $\rho\in (0,r(\epsilon)/3)$. Let $\delta:=\delta(\epsilon, \rho)$ from Definition \ref{delta}. For every $q\ge 2$ there exist $\tilde n$, $\tilde C$
%and $\delta=\delta(\epsilon, \rho)>0$
as follows.
%$\tilde C_1$ ($0<\tilde C_1<\tilde C$)
%such that

\iffalse

%Under the conditions of Lemma~\ref{key}, there is $\delta>0$ as follows.
Let $x\in E_{\epsilon, p_n, q}$. Denote for brevity
$$x^k:=x_{-kp_n}  \mbox{ and }  R^k:=R_{t(x^k)}, \ \ k=0,1,...,q-1.$$
By Lemma \ref{p0}(b), $x^k\in J_{n,x}$. Hence,
$t(x^k)\in s_{n, j_n(x)}\subset S_{n,j_n(x)}$, $0\le k\le q-1$.

\fi

For every $n>\tilde n$ consider the closed set $E_{\epsilon, p_n, q}$.
Fix $0\le i<j\le q-1$.
%Assume that
%for some $i=1,...,q$, the distance (on ${\bf S^1}$) between $t(x)$ and $\tau(t(x),x^i)$ is less than $\delta$.
%lie in the same component of $s_{n,j_n(x)}$.
%Let $I_{ij}=<t(x^i), t(x^j)>$ be one of the two arcs of ${\bf S^1}$ with end points $t(x^i), t(x^j)$ which is contained in $S_{n, j_n(x)}$. Assume that
%\begin{equation}\label{xixj}
%|x^i-x^j|<\frac{\rho}{4} \mbox{ and }
%\mbox{ the lenght of the arc } I_{ij} \mbox{ is less than } \delta.
%\end{equation}
%Moreover,
Assume for an arbitrary $n$ as above, that either (I) $t(x^j)$ and $t(x^i)$ belong to a single component of $s_{n, j_n(x)}$, or (II) the map $\sigma^{j_n(x)-1}: S_{n,1}\to S_{n, j_n(x)}$
is a homeomorphism and the length of the arc $S_{n, j_n(x)}$ is less than $\delta$.
%Let $W=W_n(x^j, \rho)$ in the first case and
%$W=\widetilde W_n(x^j, \rho)$ in the second.

Then:

{\bf (a)} the map $f^{(j-i) p_n}: g_{(j-i) p_n}(B(x^i,\rho))\to B(x^i,\rho)$ has a unique fixed point
$a=a_n$ and $a\in J_{n,x}$,

{\bf (b)} there is a semi-open simple curve
$$\gamma_{p_n,q,i,j}(x)\subset B(x^i, \rho)\cap A(\infty)$$
such that:
\begin{enumerate}
\item it lands at $a$ and $g_{(j-i) p_n}(\gamma_{p_n,q,i,j}(x))\subset \gamma_{p_n,q,i,j}(x)$.
Another end point $b$ of $\gamma_{p_n,q,i,j}(x)$ lies in $R^i$ and $G(b)>\tilde C/2$,
\item $\gamma_{p_n,q,i,j}(x)=\cup_{l\ge 0} g_{(j-i)p_n}^l(L_0\cup L_1)$ where the 'fundamental arc' $L_0\cup L_1$
consists of an arc $L_0$ of an equipotential of the level at least $\tilde C/2$ that joins a point $b\in R^i$ with a point $b_1\in R^j$,
being extended by an arc $L_1$ of the ray $R^j$ between points $b_1$ and $g_{(j-i)p_n}(b)\in R^j$;
in particular, the Green function $G(y)$ at a point $y$
is not increasing as $y$ moves from $b$ to $a$ along $\gamma_{p_n,q,i,j}(x)$,
\item the point $a$ is the landing point of a ray $R(a)$ which is fixed by $f^{(j-i) p_n}$ and which is homotopic
to $\gamma_{p_n,q,i,j}(x)$ through a family of curves in $A(\infty)$ with the fixed end point $a$.
%Note that $g_{(j-i) p_n}(b)\in R^j\cap \gamma_{p_n,q,i,j}(x)$.
\item arguments of all points of the curve $g_{(j-i) p_n}(\gamma_{p_n,q,i,j}(x))$ lie in a single component of $s^1_{n, j_n(x)}$ in the case (I)
and in a single component of $s_{n,j_n(x)}$ in the case (II) (recall that $s^1_{n, j_n(x)}$ has $4$ components and $s_{n,j_n(x)}$ has $2$ components, see Sect \ref{prel}, (C)).

%if $t(x^i), t(x^j)$ lie in a single component $s'$ of $s_{n, j_n(x)}$ then arguments of all points of $\gamma_{p_n,q,i,j}(x)$ lie in $s'$;
%if $t(x^i)\in s', t(x^j)\in s''$ where $s', s''$ are different components of $s_{n, j_n(x)}$ then arguments of all points of $g_{(j-i) p_n}(\gamma_{p_n,q,i,j}(x))$ lie in $s''$.

%\iffalse

%in the case (I), arguments of the points of $\gamma_{p_n,q,i,j}(x)$ lie in the component $s'$
%of $s_{n,j_n(x)}$ and in the case (II) arguments of the points of $g_{(j-i) p_n}(\gamma_{p_n,q,i,j}(x))$
%all lie in a single component of $s_{n,j_n(x)}$.

\end{enumerate}
Besides,
\begin{equation}\label{xja}
|a-x^j|\to 0 \mbox{ and } \log\frac{|(g_{(j-i)p_n})'(x^j)|}{|(g_{(j-i)p_n})'(a)|}\to 0
\end{equation}
as $n\to\infty$, uniformly in $x^j$ and $q$.
%$G(w(a))>\tilde C_2$, where $w(a)$ is the first intersection of the ray $R(a)$ with $\partial B(x,\rho/2)$.

{\bf (c)} if $j-i=1$ then $a=\beta_{n,j_n(x)}$ where
$\beta_{n,j_n(x)}=f^{j_n(x)-1}(\beta_n)$, the non-separating
fixed point of $f^{p_n}:J_{n,x}\to J_{n,x}$. Moreover,
$$\chi(\beta_{n,j_n(x)}):=\frac{1}{p_n}\log |(f^{p_n})'(\beta_{n,j_n(x)})|=\frac{1}{p_n}\log |(f^{p_n})'(\beta_{n})|\to \chi(\mu)$$
as $n\to\infty$.
%where $\chi(\mu)$ is the Lyapunov exponent of the measure $\mu$.
\end{lemma}
\begin{remark}
Note that $a\notin J_\infty$ while $x, x^1,...,x^{q-1}\in J_\infty$.
\end{remark}
\begin{proof}
%All $x, x^1,...,x^q\in J_{x,n}$ by combinatorics and definition of $g_n$.

\iffalse

Given $\hat r\in (0,1)$ and $\hat C>0$, define $\hat\delta=\hat\delta(r,\hat C)>0$
%and $\tilde n\in {\bf N}$
as follows. Let $C_0>0$ be so that the distance between the equipotential of level $C_0$ and $J(f)$ is bigger than $1$. Then $\hat\delta=\hat\delta(\hat r/2,\hat C)>0$ is such that for any $C\in [\tilde C, C_0]$,
%and any two points $w_1,w_2\in A(\infty)$,
if $w_1,w_2$ lie on the same equipotential $\Gamma$ of level $C$ and the difference between external arguments of $w_1,w_2$ is less than
$\hat\delta$ then the length of the shortest arc of the equipotential $\Gamma$ between $w_1$ and $w_2$ is less than $\hat r/2$.
Apply Lemma \ref{p0}(c)
with $\sigma=\rho/4$ and find the corresponding $\Delta(\rho/4)$.
%Let $\tilde C=C(\rho/2)$ where $C(\nu)$ is defined in Proposition \ref{acc}.
%Apply Claim $\tilde 1$ with $\hat r=\rho$ and $\hat C=\tilde C$ and find the corresponding $\hat\delta$.
let
$$\delta=\delta(\epsilon, \rho):=\min\{\hat\delta(\rho, C(\rho/2)), \Delta(\frac{\rho}{4})\}$$
where $C(\nu)$ is defined in Proposition \ref{acc}.

\fi

Fix $n_0$ such that, for every $n>n_0$ and $x\in E_\epsilon$, the length of each 'window'
of $s_{n, j_n(x)}$ is less than $\delta$. Therefore, for $n>n_0$, in either case (I), (II),
\begin{equation}\label{xixj}
|t(x^i)-t(x^j)|<\delta,
\end{equation}
which implies, in particular, that $|x^i-x^j|<\rho/4$.

%%%%%%%%\fi       I moved this \fi a few lines above, since this paragraph were missing !!! FP

Denote $G_n:=g_{(j-i)p_n}$ which is a holomorphic univalent function in $B(x^i,\rho)$.
Since $g_m$ are uniform contractions there is $n_1$ such that $G_n(\overline{B(x^i,\rho)})\subset B(x^i,\rho/2)$ whenever
$n>n_1$. Let $\tilde n=\max\{n_0, n_1\}$.

Let also $\tilde C=C(\rho/2)$, where $C(\nu)$ is defined in Proposition \ref{acc}.

Let $a=a_n$ be the unique fixed point of the latter map $G_n$.
We construct the curve $\gamma_{p_n,q,i,j}(x)$ to the point $a$ as follows.
First, joint a point $b\in R^i$, $G(b)=(3/4)\tilde C$, to a point $b_1\in R^j$ by an arc $L_0$ of the
equipotential $\{G(z)=(3/4)\tilde C\}$. By the choice of $\delta>0$, $L_0\subset B(x^i,\rho)$.
%hence, $L_0\subset W$.
Secondly, connect $b_1$
to the point $g_{(j-i)p_n}(b)\in R^j$ by an arc $L_1\subset R^j$. Let now
$\gamma_{p_n,q,i,j}(x)=\cup_{l\ge 0} g_{(j-i)p_n}^l(L_0\cup L_1)$.
Then properties  (1), (2) in (b) are immediate and (3) follows from general properties of conformal maps. Now, by Proposition \ref{acc}(2) and (\ref{xixj}), for all $n$ big enough,
$x^j=g_{(j-i)p_n}(x^i)\in g_{(j-i)p_n}(B(x^i,\rho))\subset B(x^i,\rho)$, moreover, the modulus of the annulus $B(x^i,\rho)\setminus g_{(j-i)p_n}(B(x^i,\rho))$
tends to $\infty$ as $n\to\infty$. Therefore, (\ref{xja}) follows from Koebe and Proposition \ref{acc}(2).

It remains to show the property (3) and that $a\in J_{n,x}$.
Consider the case (II), which is equivalent to say that the map $\sigma^{p_n}: s\to S_{n, j_n(x)}$ is a homeomorphism on each of two components $s$ of $s_{n, j_n(x)}$.
Let $\Lambda$ be the set of arguments of points of the curve $\Gamma:=g_{(j-i) p_n}(\gamma_{p_n,q,i,j}(x))$. Let $s$ be a component that contains $t(x^j)$.
Assume, by a contradiction, that $\Lambda$ contains $t$ which is in the boundary of $s$. Then $t$ is the argument of a point of $G_n^l(L_0)$, for some $l\ge 1$, hence,
$\sigma^{l(j-i)p_n}(t)$ is simultaneously the argument of a point of $L_0$ and in the boundary of $S_{n,j_n(x)}$, a contradiction.
The case (I) is similar. Property (3) is verified. In fact, we proved more: for $k=0,1,\cdots,j-i-1$, the set $\sigma^{k p_n}(\Lambda)$ is a subset of a single
(depending on $k$) component of $s_{n, j_n(x)}$ in the case (II) and a single component of $s^1_{n, j_n(x)}$ in the case (I).
This implies that all point $f^{k p_n}(a)$, $0\le k\le j-i-1$, of the cycle of $f^{p_n}$ containing $a$ belong to the closure of $U_{n,j_n(x)}$ in the case (II) and to the closure
of $U_{n, j_n(x)-p_n}$ in the case (I). Therefore, this cycle lies in $J_{n,x}$, in particular, $a\in J_{n,x}$.

Proof of (c): if $j-i=1$ then
$a$ is a fixed point of $f^{p_n}:J_{n,x}\to J_{n,x}$ and, moreover, the ray $R(a)$ lands at $a$ and is fixed by $f^{p_n}$.
Hence, the rotation number of $a$ w.r.t. the map $f^{p_n}:J_{n,x}\to J_{n,x}$ is zero. On the other hand, $\beta_{n,j_n(x)}$
is the only such a fixed point, i.e., $a=\beta_{n,j_n(x)}$ as claimed.
Then (\ref{xja}) implies that $\chi(\beta_{n,j_n(x)})\to \chi(\mu)$.
\end{proof}

{\bf For the rest of the paper}, let us {\bf fix $Q$, $\epsilon$, $r$, $\rho$, $\tilde n$, $\tilde C$ and} $\delta$ as follows:

{\it $Q\in\N$, $Q>3$, is such that}
$$Q>4\log 2/\chi(\mu).$$
This choice is motivated by the following

{\bf Fact} (\cite{Pom}, \cite{Lineq}, \cite{H}): if a repelling fixed point $z$ of $f^n$ is the landing point of $q$
rays, then $\chi(z):=(1/n)\log |(f^n)'(z)|\le (2/q)\log 2$. Hence, if $\chi(z)>\chi(\mu)/2$, then
$q<Q$.

Furthermore, {\it fix $\epsilon>0$ such that $2^{100}Q\epsilon<1$,
apply Proposition~\ref{acc} and Lemma~\ref{fixed} and find, first, $r=r(\epsilon)$, then
fix $\rho\in (0,r/32)$ and find the corresponding $\tilde n$, $\tilde C$ and $\delta$.}

\

Let
$$X_n=E_{\epsilon,p_n,4}\cap E_{\epsilon, p_{n-1},Q}=\cap_{i=0}^3 f^{i p_n}(E_\epsilon)\cap_{k=0}^{Q-1} f^{k p_{n-1}}(E_\epsilon).$$
%For every $n$ consider the set of indices (that could be empty)
%$$I_n=\{0<j<p_n: \tilde t_{n,j}-t_{n,j}<\delta\}.$$
%Note that if for some large $n$
%$$\frac{|I_{n-1}|}{p_{n-1}}>2^{-6},$$
%then
%$$\frac{|L_{n-1}\cap I_{n-1}|}{p_{n-1}}>2^{-7}.$$
Let us analyze several possibilities.

\begin{lemma}\label{p1}
%Let $j\in L_{n-1}$ and let $x=x^0\in X_n\cap f^j(J_{n-1})$.
There is $n_*>\tilde n$ as follows. Let $n>n_*$ and $x\in X_n$.
Consider $J_{n,x}=f^{j_n(x)}(J_n)\subset f^{j_{n-1}(x)}(J_{n-1})$ so that $x\in J_{n,x}$.

%Let $x=x^0\in X_n$ so that $x\in J_{n,x}=f^{j_n(x)}(J_n)$.

Let $x^0=x$ and $x^1=x_{-p_n}$. Assume that either
(I) $t(x^0)$, $t(x^1)$ belong to a single component of $s_{n, j_n(x)}$, or (II) the map $\sigma^{j_n(x)-1}: S_{n,1}\to S_{n, j_n(x)}$
is a homeomorphism and the length of the arc $S_{n, j_n(x)}$ is less than $\delta$.

Then:

(i)
$\chi(\beta_{n,j_n(x)})=\chi(\beta_n)\to \chi(\mu)$ as $n\to \infty$ and $\chi(\beta_n)>\chi(\mu)/2$ for $n>n_*$.

(ii) assume that $f^{p_n}$ is satellite, i.e., (by Lemma \ref{satellite}) $\beta_n$ has period $p_{n-1}$, $q_n\ge 2$ with rotation number $r_n/q_n$ of $\beta_n$, and $\beta_{n,j_n(x)}$
is the $\alpha$ (i.e., separating) fixed point of $f^{p_{n-1}}: J_{n-1,x}\to J_{n-1,x}$.
Then $q_n<Q$ and
\begin{equation}\label{xkbeta}
|\beta_{n,j_n(x)}-x_{-k p_{n-1}}|\to 0, \ n\to\infty, \mbox{ uniformly in } x\in X_n, \ 1\le k\le q_{n}.
\end{equation}
There exist two simple semi-open curves $\gamma(x)$ and $\tilde\gamma(x)$ that satisfy the following properties:
\begin{enumerate}
\item $\gamma(x)$ and $\tilde\gamma(x)$ tend to $\beta_{n,j_n(x)}$ and
$\gamma(x), \tilde\gamma(x)\subset B(x^0, \rho)\cap A(\infty)$,
\item $\gamma(x), \tilde\gamma(x)$ consist of arcs of equipotentials and external rays;
the starting point $b_1=b_1(x)$ of $\gamma(x)$ lies in an arc of $R_{t(x^1)}$
and the starting point $\tilde b_1=\tilde b_1(x)$ of $\tilde\gamma(x)$ lies in an arc of $R_{t(\tilde x)}$
where $\tilde x=x_{-i p_{n-1}}$ for some $i=i(x)\in\{1,\cdots, q_n-1\}$, such that levels of $b_1$ and $\tilde b_1$ are
equal and at least $\tilde C/4$,
\item one of the two curves (say, $\gamma(x)$) is homotopic, through curves in $A(\infty)$ tending to $\beta_{n,j_n(x)}$,
to the ray $R_{t_{n,j_n(x)}}=f^{j_n(x)-1}(R_{t_n})$, and another one - to the ray $R_{\tilde t_{n,j_n(x)}}=f^{j_n(x)-1}(R_{\tilde t_n})$;
%???????ne nado???????? furthermore, both curves are invariant under the branch $g_{p_n}$ of $f^{-p_n}$ that fixes $\beta_{n,j_n(x)}$,???
\item $\gamma(x)$, $\tilde\gamma(x)\subset U_{n-1, j_{n-1}(x)}$,
%arguments of points of $\gamma(x)$ lie in the interval $[t_{x^1},t_{n,j_n(x)})$ and arguments of points of $\tilde\gamma(x)$ - in $[t_{\tilde{x}},\tilde t_{n,j_n(x)})$,??????????????????
%where the lengths of these two intervals tend to zero as $n\to\infty$ uniformly in $j$ and $x$,
\item $\gamma(x)\subset U_{n,j_n(x)}$, $\tilde\gamma(x)\subset U_{n,j_n(\tilde x)}$, in particular, $\gamma(x),\tilde\gamma(x)$ are disjoint; being completed by their common limit point $\beta_{n,j_n(x)}$ and two other arcs: an arc of the
ray $R_{t(x^1)}$ from $b_1\in\gamma(x)$ to $\infty$ and an arc of the ray $R_{t(\tilde x)}$ from $\tilde b_1\in\tilde\gamma(x)$ to $\infty$,
they split the plane into two domains such that one of them contains $I:=J_{n,x}\setminus \beta_{n,j_n(x)}$
and another one contains all $q_n-1$ other different iterates $f^{k p_{n-1}}(I)$, $1\le k\le q_n-1$.
The intersection of closures of all those $q_n$ sets consists of the fixed point $\beta_{n,j_n(x)}$ of $f^{p_{n-1}}$.
%which is at the same time $\alpha$-fixed point of $f^{p_{n-1}}: J_{n-1,x}\to J_{n-1,x}$.
\end{enumerate}
\end{lemma}
\begin{remark}
Beware that the point $x$ that determines both curves $\gamma(x)$, $\tilde\gamma(x)$ does not belong to either of these curves.
\end{remark}

\iffalse

\begin{figure}[h]
\includegraphics[width=7cm,height=7cm]{CCF_000009.pdf}
\caption{Lemma \ref{p1}}
\label{fig2}
\end{figure}

\fi

\begin{proof}
(i) follows from Lemma~\ref{fixed} where we take $i=0, j=1$. Fix $n_*>\tilde n$ such that $\chi(\beta_n)>\chi(\mu)/2$ for all $n>n_*$.

%(II) since $\tilde t_{n,j_n(x)}-t_{n,j_n(x)}<t_{n-1,j}-t_{n-1,j}<\delta$.

Let us prove (ii). Here we build a "flower" of arcs at the $\beta$ fixed of the satellite $f^{p_n}$ starting with an arc which is fixed by $f^{p_n}$
and then "rotate" this arc by a branch of $f^{-p_{n-1}}$ (for which the same $\beta$ point is also a fixed point, see (C)).
Let $\gamma'(x):=\gamma_{p_n,1,0,1}(x)$ where the latter curve is defined in Lemma~\ref{fixed}.
Then properties (1)-(3) of the curve $\gamma(x)$ are satisfied also for $\gamma'(x)$. In particular, $\gamma'(x)$ is homotopic to $R_{t_{n,j_n(x)}}$.

%note that by Corollary~\ref{fixed} there exists a curve
%$\gamma(x)$ with the listed properties, except, perhaps, the last one. In particular, $\gamma(x)$ is homotopic to $R_{t_{n,j_n(x)}}$.

As both $\tilde t_{n,j_n(x)}, t_{n,j_n(x)}$ are external arguments of $\beta_{n,j_n(x)}$ which is a $p_{n-1}$-periodic point of $f$, there is $i\in\{1,\cdots, q_n-1\}$ such that $\sigma^{i p_{n-1}}(\tilde t_{n,j_n(x)})=t_{n,j_n(x)}$.
%landing at $\beta_{n,j_n(x)}$ and with another end point $b$ such that
%$G(b)>\tilde C$. The curve $\Gamma$ is homotopic in $A(\infty)$ to one of the rays
%of arguments $t_{n,j_n(x)}, \tilde t_{n,j_n(x)}$ among the curved that fix $\beta_{n,j_n(x)}$.
%Assume for definicity that the argument is $t_{n,j_n(x)}$.
%Secondly, since $\chi(\beta_n)>(1/2)\chi(\mu)$ and by the choice of $Q$, $q_n<Q$.
Now we use that $x\in E_{\epsilon, p_{n-1},Q}$ and that $q_{n}<Q$ to prove (\ref{xkbeta}).
%and apply Corollary~\ref{fixed} with $p=e_n p_{n-1}$.
%Note that $f^{p}(\beta_{n,j_n(x)})=\beta_{n,j_n(x)}$ and
%$x_{-p_n}\in g_{p_n}(B(x,\rho))$.
%Let us assume the contrary: $q_n>1$. Then $f^{p_{n-1}}(\beta_{n,j_n(x)})=\beta_{n,j_n(x)}$.
%First of all, let us show that
%\begin{equation}\label{xkbeta}
%|\beta_{n,j_n(x)}-x_{-k p_{n-1}}|\to 0, \ \  k=1,\cdots,q_{n-1}-1, \mbox{ as } n\to\infty.
%\end{equation}
%as $n\to\infty$, for all $k\in\{1,cdots,q_{n-1}-1\}$.
Indeed, for each $k=\{1,\cdots, q_n\}$, since $f:J_\infty'\to J_\infty'$ is a homeomorphism and
$x_{-k p_{n-1}}\in E_\epsilon$, we have: $g_{p_n}=g_{(q_{n}-k) p_{n-1}}\circ g_{k p_{n-1}}$.
Hence, if $\beta'=g_{k p_{n-1}}(\beta_{n,j_n(x)})$, then $\beta_{n,j_n(x)}=g_{(q_{n-1}-k) p_{n-1}}(\beta')$
implying that $\beta'=f^{(q_{n}-k) p_{n-1}}(\beta_{n,j_n(x)})=\beta_{n,j_n(x)}$.
Then $\beta_{n,j_n(x)}, x_{-k p_{n-1}}\in g_{k p_{n-1}}(B(x,\rho))$ which along with Proposition \ref{acc}, part (2)
imply (\ref{xkbeta}).

In turn, (\ref{xkbeta}) implies that, provided $n$ is big,
%$g_{p_{n-1}}(x)\in B(x,\rho)$ and moreover
$g_{k p_{n-1}}: B(y,\rho/2)\to B(y,\rho/2)$ uniformly in $k=0,1,\cdots, q_n$ where $y$ is either $\beta_{n,j_n(x)}$ or $x_{-k p_{n-1}}$.

Now we consider a curve $g_{i \tilde p_{n}}(\gamma'(x))$ that starts at $x_{-i \tilde p_{n}}$ and tends to $\beta_{n,j_n(x)}$.
% $g_{i p_{n-1}}(\gamma'(x)$ that starts at $x_{-i p_{n-1}}$, and $g_{q_n p_{n-1}}(\gamma'(x)$ that starts at $x_{-p_n}$ ($p_n=q_n p_{n-1}$ of course).
By Proposition \ref{acc} coupled with (\ref{xkbeta}), one can join $x_{-i p_{n-1}}$ by an arc of the ray $R_{t(x_{-i p_{n-1}})}$
{\it inside of} $B(x,\rho/2)$ up to a point of level $\tilde C/4$. This will be the required curve $\tilde\gamma(x)$.
To get the curve $\gamma(x)$ we modify $\gamma'(x)=\gamma_{p_n,1,0,1}(x)=\cup_{l\ge 0} g_{p_{n}}^l(L_0\cup L_1)$
by cutting off the arc $L_0$ of an equipotential: $\gamma(x)=\gamma'(x)\setminus L_0$ (see Lemma \ref{fixed} for details about $L_0$).
Properties (1)-(5) follow.

%Another way to get the same curve is similar as we construct $\tilde\gamma(x)$: consider $g_{q_n p_{n-1}}(\gamma'(x))$ and extend it by an arc
%of $R_{t(x_{-p_n})}$.

%build inductively $i+1$ curves applying consequently
%$$g_{p_{n-1}}, g_{2p_{n-1}},..., g_{i p_{n-1}}$$
%to the previous curve starting with $\gamma(x)$ and extending each time the image by an arc of the ray $R_{t(x_{-p_{n-1}-k p_{n-1}})}$
%{\it inside of} $B(x,\rho/2)$ up to a point of level $\tilde C/4$. This is possible
%by Proposition~\ref{acc}. We
%obtain $i+1$ different curves in $B(x,\rho/2)\cap A(\infty)$
%each of which joins the point $\beta_{n,j_n(x)}$ with a point
%of equipotential of the level at least $\tilde C/4$..
%Moreover, the last of these curves, called $\tilde\gamma(x)$, is homotopic to the ray of argument
%$\tilde t_{n,j_n(x)}$. Then all listed properties are satisfied for $\tilde\gamma(x)$, by the construction.
%$It remains to modify the curve $\gamma(x)=\gamma_{p_n,1,0,1}(x)$ to satisfy the last properties, as follows:
%cut off from $\gamma_{p_n,1,0,1}(x)$ the arc of equipotential between points $b$ and $b_1$, see Corollary~\ref{fixed}.

\end{proof}
%We need the following notions.

Given a point $x=x^0$ and $n$ such that $x\in f^{j}(J_n)\cap E_{\epsilon, p_n, 1}$, where $j=j_n(x)$, let $x^1=x_{-p_n}$ and $t(x^0)$, $t(x^1)$
the arguments of $x^0$, $x^1$ as in Proposition \ref{acc}.
We call $x$ $n$-{\bf friendly} if
$t(x^0)$ and $t(x^{1})$ lie in the same component of $s_{n,j}$ and $n$-{\bf unfriendly} otherwise
(or simply friendly and unfriendly if $n$ is clear from the context).
The name reflects the fact that for an $n$-friendly point $x$ the condition (I) of Lemma \ref{p1} always holds
for $x^1=x$ and $x^2=x_{-p_n}$, so Lemma \ref{p1} always applies.

When the rotation number of $\alpha_n$ is equal to $1/2$ we have:
\begin{lemma}\label{p3}
There is $\tilde C_3>0$ (depending only on fixed $\epsilon$ and $\rho$) as follows.
Suppose that, for some $n>\tilde n$, the rotation number
of the separating fixed point $\alpha_n$ is equal to $1/2$.
Let $z=z^0\in f^j(J_n)\cap E_{\epsilon, p_n, 3}$ and $z^i=z_{-i p_n}$, $i=1,2,3$.
%we call $x^i$ ($0\le i<3$) {\it friendly} if
%$t^i(x)$ and $t^{i+1}(x)$ lie in the same component of $s_{n,j}$ and {\it unfriendly} otherwise.
%For $0<j<p_n$, let $UF_{n,j}$ be the set of $x\in E_{\epsilon, p_n,3}\cap f^j(J_n)$ such that
Assume that all three points $z^0, z^1, z^2$ are $n$-unfriendly.

%Assume that the set $UF_{n,j}$ is not empty and let $x\in UF_{n,j}$ and
%$x^k=x_{-k p_n}$, $k=0,1,2,3$.
%Assume that $|t(x^0)-t(x^2)|<\delta$, $|t(x^1)-t(x^3)|<\delta$.
Then
there exist two (semi-open) curves $\gamma^{1/2}_n(z)$ and $\tilde\gamma^{1/2}_n(z)$
consisting of arcs of rays and equipotentials with the following properties:

(i) $\gamma^{1/2}_n(z)\subset B(z,\rho)$, $\tilde\gamma^{1/2}_n(z)\subset B(z^1,\rho)$, moreover, arguments of points of $\gamma^{1/2}_n(z)$ lie in one 'window' of $s_{n,j}$ while
arguments of points of $\tilde\gamma^{1/2}_n(x)$ lie in another 'window' of $s_{n,j}$,

(ii) $\gamma^{1/2}_n(z)$ and $\tilde\gamma^{1/2}_n(z)$ converge to a common point $\alpha^*_{n,j}$ which is a fixed point of $f^{p_n}:f^j(J_n)\to f^j(J_n)$
(i.e., $\alpha^*_{n,j}$ is either the non-separating fixed point $\beta_{n,j}$ or the separating fixed point $\alpha_{n,j}$,

%(iii) arguments of the points of $\Gamma_n(x)$ lie in one window of $s_{n,j}$ while arguments of the points of $\tilde\Gamma_n(x)$ lie in another window of $s_{n,j}$,

(iii) starting points of $\gamma^{1/2}_n(z), \tilde\gamma^{1/2}_n(z)$
have equal Green level which is bigger than $\tilde C_3$,

(iv) $z^k-\alpha^*_{n,j}\to 0$, $0\le k\le 3$, as $n\to\infty$.
\end{lemma}

\begin{proof}
%Recall notations $x^i=x_{-i p_{n-1}}$.
As $z\in E_\epsilon$, lengths of 'windows' of $s_{n,j_n(z)}$ tend uniformly to zero as $n\to\infty$. It follows from the definition of friendly-unfriendly points that
$t(z^0),t(z^2)$ are in one 'window' of $s_{n,j}$ and $t(z^1),t(z^3)$ are in another 'window' of $s_{n,j}$. Therefore, condition (I) of Lemma \ref{fixed} holds for each pair $z^0,z^2$ and $z^1,z^3$.
Now, apply Lemma \ref{fixed} to $z\in E_{\epsilon,p_n,3}$, first, with $i=0$, $j=2$, and then
with $i=1$, $j=3$. Let $\gamma^{1/2}_n(z)=\gamma_{p_n,3,0,2}(z)$ and $\tilde\gamma^{1/2}_n(z)=\gamma_{p_n,3,1,3}(z)$.
Then (i),(iii) hold.
To check (ii), note that these curves converge to some points $\alpha$,$\tilde\alpha\in f^j(J_n)$
which are fixed by $f^{2p_n}$
%though different from the non-separating fixed point $\beta_{n,j}$
%of $f^{p_n}:f^j(J_n)\to f^j(J_n)$ because arguments of $t(x),t(x^1)$ lie in different windows
%of $s_{n,j}$ and arguments of $t(x^1),t(x^2)$ lie in different windows of $s_{n,j}$ as well.
%Hence, these arguments form a 2-cycle for the map $\sigma^{p_n}$.
On the other hand, since the rotation number of $\alpha_n$ is $1/2$, $f^{p_n}:f^j(J_n)\to f^j(J_n)$ has no $2$-cycle.
%those external rays must tend
%to the $\alpha$-fixed point $\alpha_{n,j}$ of $f^{p_n}:f^j(J_n)\to f^j(J_n)$.
%the map $f^{p_n}:f^j(J_n)\to f^j(J_n)$ has no cycle of period $2$.
Therefore, one must have either $\alpha=\tilde\alpha=\beta_{n,j}$ or
$\alpha=\tilde\alpha=\alpha_{n,j}$, i.e., (ii) holds too. As $t(z^0)-t(z^2)\to 0$ and $t(z^1)-t(z^3)\to 0$ as $n\to\infty$, $z^0-z^2, z^1-z^3\to 0$, too, by Lemma \ref{p0}.
Besides, by (\ref{xja}), $z^2-\alpha, z^3-\tilde\alpha\to 0$ as $n\to\infty$. As $\alpha=\tilde\alpha=\alpha^*_{n,j}$,
(iv) also follows.
\end{proof}
The following is a consequence of Lemmas \ref{fixed} and \ref{p3}:
\begin{lemma}\label{doublem}
Let $n>\tilde n$. Assume that $f^{p_n}$ is satellite and doubling, i.e., $\beta_n=\alpha_{n-1}$ and the rotation number of $\alpha_{n-1}$ is equal to $1/2$
(in particular, $p_n=2p_{n-1}$).
For some $1\le j\le p_{n-1}$, denote $J:=f^j(J_{n-1})$. Let $J^1:=f^j(J_{n})$, $J^0:=f^{j+p_{n-1}}(J_{n})$ be the two small Julia sets of the next level $n$ which are
contained in $J$ (note that $J^0$ contains the critical point and $J^1$ contains the critical value of the map $F:=f^{p_{n-1}}: J\to J$).
Let $x\in J^1\cap E_\epsilon$ be such that all its $5$ forward iterates $x_{kp_{n-1}}=F^k(x)\in E_\epsilon$, $k=1,2,3,4,5$. Then there exist two simple semi-open curves
$\Gamma_n^{1/2}(x)$, $\Gamma_n^{1/2}(x)$ consisting of arcs of rays and equipotentials that satisfy essentially conclusions of the previous lemma where $n$ is replaced by $n-1$, i.e.:

(i) $\Gamma^{1/2}_n(x), \tilde\Gamma^{1/2}_n(x)\subset B(x,3/2\rho)$, moreover, arguments of points of $\Gamma^{1/2}_n(x)$ lie in one 'window' of $s_{n-1,j_{n-1}(x)}$ while
arguments of points of $\tilde\Gamma^{1/2}_n(x)$ lie in another 'window' of $s_{n-1,j_{n-1}(x)}$,

(ii) $\Gamma^{1/2}_n(x)$ and $\tilde\Gamma^{1/2}_n(x)$ converge to a common point $\beta^*_{n-1,j_{n-1}(x)}$ which is a fixed point of $f^{p_{n-1}}: f^j(J_{n-1})\to f^j(J_{n-1})$
(i.e., $\beta^*_{n-1,j_{n-1}(x)}$ is either the non-separating fixed point $\beta_{n-1,j_{n-1}(x)}$ or the separating fixed point $\alpha_{n-1,j_{n-1}(x)}$,

%(iii) arguments of the points of $\Gamma_n(x)$ lie in one window of $s_{n,j}$ while arguments of the points of $\tilde\Gamma_n(x)$ lie in another window of $s_{n,j}$,

(iii) starting points of $\Gamma^{1/2}_n(x), \tilde\Gamma^{1/2}_n(x)$
have equal Green level which is bigger than $\tilde C_3$,

(iv) $x_{k p_{n-1}}-\beta^*_{n-1,j_{n-1}(x)}\to 0$, $0\le k\le 3$ as $n\to\infty$ uniformly in $x$.
\end{lemma}
\begin{remark}
Condition $F^k(x)\in E_\epsilon$, $0\le k\le 5$, is equivalent to the following: $x\in f^{-5p_{n-1}}(E_{\epsilon,p_{n-1},6})$.
\end{remark}
\begin{proof}
To fix the idea let's replace $f^{p_{n-1}}: f^j(J_{n-1})\to f^j(J_{n-1})$, using a conjugacy to a quadratic polynomial,
by a quadratic polynomial (denoted also by $F$) so that now $F:J\to J$ where $J=J(F)$ and $F^2$ is satellite with two small Julia sets $J^0$, $J^1$ that
meet at the $\alpha$-fixed point of $F$ and rays of arguments $1/3$, $2/3$ land at $\alpha$ . Here $0\in J^0$, $F(0)\in J^1$, $F: J^1\to J^0$ is a homeomorphism while
$F:J^0\to J^1$ is a two-to-one map. If a ray $R_t$ of $F$ has its accumulation set in $J^1$ then $t\in [1/3,5/12]\cup [7/12,2/3]$ and if $R_t$ accumulates in $J^0$ then
$t\in [1/6,1/3]\cup [2/3,5/6]$. This implies that if $R_t$ lands at $x\in J^1$ and $t$ lies in one of the two 'windows' $[0,1/2)$, $(1/2, 1]$ then
$R_{\sigma(t)}$ lands at $J^0$ where $\sigma(t)$ must be in a different 'window' (in other words, points of $J^0$ are 'unfriendly').
%$t\in [1/3,5/12]\subset [0,1/2]$ (respectively, $t\in [7/12,2/3]\subset [1/2,1]$) then
%$R_{\sigma(t))$ lands at $F(x)\in J^0$ and $\sigma(t)\in [2/3,5/6]\subset [1/2,1]$ (respectively, $\sigma(t)\in [1/6,1/3]\subset [0,1/2]$)
Coming back to $f^{p_{n-1}}$ this means that, for $x\in J^1$, $t(x),t(F(x))$ are always in different components (where by 'component' we mean a component of $s_{n-1,j}$).
Besides, for $y\in J_\infty\cap J$, $y$ and $F(y)$ are always in different $J^i$, $i=0,1$.
This leaves us with the only possibilities:

(i) $t(F(x)), t(F^2(x))$ are in different components;
this implies that $t(x), t(F(x))$ are in different components and $t(F(x)), t(F^2(x))$ are in different components,
that is, points $F^3(x), F^2(x), F(x)$ are all unfriendly;

(ii) $t(F(x)), t(F^2(x))$ are in the same components; there are two subcases:

(ii') $t(F^3(x)), t(F^4(x))$ are in different components, i.e., (i) holds with $x$ replaced by $F^2(x)$
which implies that $F^5(x), F^4(x), F^3(x)$ are all unfriendly;

(ii'') $t(F^3(x)), t(F^4(x))$ are in the same component which then means that $F^2(x)$ and $F^4(x)$ are both friendly.

In the case (i) and (ii'), apply Lemma \ref{p3} with $n-1$ instead of $n$ to $z=F^3(x)$ and to $z=F^5(x)$, respectively,
letting $\Gamma^{1/2}_n(x)=\gamma^{1/2}_{n-1}(F^3(x))$, $\tilde\Gamma^{1/2}_n(x)=\tilde\gamma^{1/2}_{n-1}(F^3(x))$ and
$\Gamma^{1/2}_n(x)=\gamma^{1/2}_{n-1}(F^5(x))$, $\tilde\Gamma^{1/2}_n(x)=\tilde\gamma^{1/2}_{n-1}(F^5(x))$, respectively.
In the case (ii''), apply Lemma \ref{fixed} with $p_{n-1}$, $q=1$, $i=0,j=0$, first, to the point $F^2(x)$ and then to the point $F^4(x)$
letting $\Gamma^{1/2}_n(x)=\gamma_{p_{n-1},1,0,1}(F^2(x))$, $\tilde\Gamma^{1/2}_n(x)=\gamma_{p_{n-1},1,0,1}(F^4(x))$.
\end{proof}

\section{Proof of Theorem~\ref{thm:meas+}}\label{thm:proof}
Every invariant probability measure with positive Lyapunov exponent has an ergodic component with positive exponent.
So let $\mu$ be such an ergodic $f$-invariant measure component supported in $J_\infty$.
First, we have the following general
\begin{remark}\label{genrem} Given $x\in J'_\infty$ such that $\tilde r(x)>0$ as in Proposition \ref{lyplus}, and given $n$,
%and $0<j_n(x)<p_n$ so that $K_x\subset f^{j_n(x)}(J_n)$,
the set $J_{n,x}=f^{j_n(x)}(J_n)$ cannot be covered by $B(x,\tilde r(x))$ because
otherwise the branch $g_{p_n}: B(x,\tilde r(x))\to {\bf C}$
of $f^{-p_n}$, which sends $x$ to $x_{-p_n}\in J_{n,x}$ meets the critical value along the way so cannot be well-defined.
Thus $\diam J_{n,x}>\tilde r(x)$, for each $n$, and $\diam K_x=\lim \diam J_{n,x}\ge \tilde r(x)$.
In particular, $\diam J_{n,x}\ge r(\epsilon)$ for all $x\in E_\epsilon$ and $n$.
\end{remark}
We need to prove that $f$ has finitely many satellite renormalizations.
Assuming the contrary, let $\mathcal{S}$ be an infinite subsequence such that $f^{p_n}$ is a satellite renormalization of
$f$ for each $n\in\mathcal{S}$.

%\marginpar{The rest is rewritten NN}

We arrive at a contradiction by considering, roughly speaking, two alternative situations. In the first one, we find a point $x\in E_\epsilon$, $n$, and two curves in $B\cap A(\infty)$
where $B:=B(x,\tilde r(x))$
that tend to the $\beta$-fixed points of $J_{n,x}$ such that another ends of the curves can be joined by an arc of equipotential in $B$ thus 'surrounding'
$J_{n,x}$ by a 'triangle' in $B$ which would be a contradiction as in Remark \ref{genrem}. The second situation is when the first one does not happen. Then we use several curves to 'surround'
$J_{n,x}$ by a 'quadrilateral' in $B$, ending by the same conclusion. The curves we use have been constructed in Lemmas \ref{p1}, \ref{doublem}.

The first situation happens in cases A and B1, and the second one in B2.

{\bf Case A}: $\mathcal{S}$ {\it contains an infinite sequence of indices of non-doubling renormalizations}.
Passing to a subsequence one can assume that $f^{p_n}$ is satellite not doubling for every $n\in\mathcal{S}$.

Fix $\zeta=1/4$. By Lemma \ref{notdoubl}, for each $n\in\mathcal{S}$ and each $j=1,\cdots,[\zeta p_n]$, the map
$\sigma^{j-1}:S_{n,1}\to S_{n,j}$ is a homeomorphism and the length $|S_{n,j}|\to 0$ as $n\to\infty$ uniformly in $j$.
Fix $N$ such that $|S_{n,j}|<\delta$ for each $n>N$, $n\in\mathcal{S}$.
For $n\in\mathcal{S}$, let
$$\mathcal{C}_n=\{f^j(J_n)| 1\le j\le [\zeta p_n]\}.$$
Let $n,m\in\mathcal{S}$, $m<n$. Denote $p=p_m, \tilde p =p_n$, $q=p_n/p_m$. The intersection
$\mathcal{C}_n\cap \mathcal{C}_m$ contains all $f^{j+kp}(J_n)$ with $1\le j\le [\zeta p]$, $j+kp\le [\zeta \tilde p]$. Hence,
$$\#(\mathcal{C}_n\cap \mathcal{C}_m)\ge \sum_{j=1}^{[\zeta p]} [\zeta q - \frac{j}{p}]\ge [\zeta q-1][\zeta p]\ge$$
$$\tilde p (\frac{\zeta p-1}{p} \frac{\zeta q-1}{q}-\frac{\zeta}{q})\sim \zeta^2 \tilde p$$
as $p,q\to\infty$. Therefore, fixing $\kappa=\zeta^2/2$=1/8, there are $m_0,k_0$ such that for each $n,m\in\mathcal{S}$, $m>m_0$, $n>m+k_0$,
$$\mu(\mathcal{C}_n\cap \mathcal{C}_m)>\kappa.$$
Fix such $n,m$, assume also that $m>\max\{N, n_*\}$ where $n_*$ is defined in Lemma \ref{p1} and recall the set
$$X_n=E_{\epsilon,p_n,4}\cap E_{\epsilon, p_{n-1},Q}=\cap_{i=0}^3 f^{i p_n}(E_\epsilon)\cap_{k=0}^{Q-1} f^{k p_{n-1}}(E_\epsilon).$$
Since $\mu(X_n)>1-(Q+4)\epsilon>1-\kappa$, there is $x\in X_n\cap\mathcal{C}_n\cap \mathcal{C}_m$ and, by the choice of $n$, the assumption (II) of Lemma \ref{p1} holds for $x$.
Therefore, there exist two simple semi-open curves $\gamma(x)$ and $\tilde\gamma(x)$ that satisfy the following properties:
$\gamma(x)$ and $\tilde\gamma(x)$ tend to $\beta_{n,j_n(x)}$,
$\gamma(x), \tilde\gamma(x)\subset B(x, \rho)\cap A(\infty)$ and
$\gamma(x), \tilde\gamma(x)$ consist of arcs of equipotentials and external rays;
the starting point $b_1$ of $\gamma(x)$
and the starting point $\tilde b_1$ of $\tilde\gamma(x)$ have equal levels which is
at least $\tilde C/4$;
%\item one of the two curves (say, $\gamma(x)$) is homotopic, through curves in $A(\infty)$ tending to $\beta_{n,j_n(x)}$,
%to the ray $R_{t_{n,j_n(x)}}=f^{j_n(x)-1}(R_{t_n})$, and another one - to the ray $R_{\tilde t_{n,j_n(x)}}=f^{j_n(x)-1}(R_{\tilde t_n})$;
%???????ne nado???????? furthermore, both curves are invariant under the branch $g_{p_n}$ of $f^{-p_n}$ that fixes $\beta_{n,j_n(x)}$,???
$\gamma(x)$, $\tilde\gamma(x)\subset U_{n-1, j_{n-1}(x)}$; finally,
%arguments of points of $\gamma(x)$ lie in the interval $[t_{x^1},t_{n,j_n(x)})$ and arguments of points of $\tilde\gamma(x)$ - in $[t_{\tilde{x}},\tilde t_{n,j_n(x)})$,??????????????????
%where the lengths of these two intervals tend to zero as $n\to\infty$ uniformly in $j$ and $x$,
being completed by their common limit point $\beta_{n,j_n(x)}$ and arcs of rays from $b_1\in\gamma(x)$ to $\infty$ and from $\tilde b_1\in\tilde\gamma(x)$ to $\infty$,
they split the plane into two domains such that one of them contains $I:=J_{n,x}\setminus \beta_{n,j_n(x)}$ and another one contains all other iterates
$f^{kp_{n-1}}(I)$, $1\le k\le q_n-1$.
Now, since $U_{n-1, j_{n-1}(x)}\subset U_{m, j_{m}(x)}$ and by the choice of $m$, the distance between arguments of the points $b_1$ and $\tilde b_1$
inside of $S_{n-1,j_{n-1}(x)}$ is less than $\delta$. By the definition of $\delta$,
$b_1$ and $\tilde b_1$ can be joined by an arc $A_n$ of equipotential inside of $B(x,\rho)\cap U_{n-1, j_{n-1}(x)}$.
Consider a Jordan domain $Z_n$ with the boundary to be the arc $A_n$ and semi-open curves $\gamma(x)$, $\tilde\gamma(x)$ completed by their
common limit point $\beta_{n,j_n(x)}$. Then $Z_n\subset B(x,\rho)$.
By the properties of the curves , $Z_n\cup \beta_{n,j_n(x)}$ contains either $J_{n,x}$ or its iterate $f^{kp_{n-1}}(J_{n,x})$, for some $1\le k\le q_n-1$, in a contradiction with Remark \ref{genrem}.

Complementary to A is

{\bf Case B}: {\it for all big $n$, every satellite renormalization $f^{p_n}$ is doubling}, i.e., $\beta_n=\alpha_{n-1}$ and $p_n=2p_{n-1}$ for every $n\in\mathcal{S}$.

Let $Y_{n-1}=E_{\epsilon,p_{n-1},6}$ and $\tilde Y_{n-1}=f^{-5p_{n-1}}(Y_{n-1})$.
Note that $\mu(Y_{n-1})=\mu(\tilde Y_{n-1})>1-6\epsilon$.

For every $n\in\mathcal{S}$, let
$$L_{n}=\{0<j<p_{n-1}| \mu(f^j(J_{n-1})\cap \tilde Y_{n-1})>\frac{1-2^{12}\epsilon}{p_{n-1}}\}.$$
As $\mu(\tilde Y_{n-1})>1-6\epsilon$, it follows,
%for each $n\in\mathcal{S}$,
$$\# L_{n}>(1-3/2^{11})p_{n-1}.$$
Since we are in case B, each $f^j(J_{n-1})$ contains precisely two small Julia sets $f^j(J_n),f^{j+p_{n-1}}(J_n)$ of the next level $n$ each of them of measure
$1/(2p_{n-1})$. Hence, the measure of intersection of each of these small Julia sets with $\tilde Y_{n-1}$
is bigger than $(1/2-2^{10}\epsilon)/p_{n-1}>0$.
By Lemma \ref{doublem}, choosing for every $j\in L_n$ a point $x_j\in f^j(J_{n-1})\cap \tilde Y_{n-1}$ we get a pair of curves
$\Gamma^{1/2}_{n}(x_j),\tilde\Gamma^{1/2}_n(x_j)$ consisting of arcs of rays and equipotentials
as follows:
(i) $\Gamma^{1/2}_{n}(x_j), \tilde\Gamma^{1/2}_{n}(x_j)\subset B(x_j,3/2\rho)$, moreover, arguments of points of $\Gamma^{1/2}_{n}(x_j)$ lie in one 'window' of $s_{n-1,j}$ while
arguments of points of $\tilde\Gamma^{1/2}_{n}(x_j)$ lie in another 'window' of $s_{n-1,j}$,
(ii) $\Gamma^{1/2}_{n}(x_j)$ and $\tilde\Gamma^{1/2}_{n}(x_j)$ converge to a common point $\beta^*_{n-1,j}$ which is a fixed point of $f^{p_{n-1}}: f^j(J_{n-1})\to f^j(J_{n-1})$
(i.e., $\beta^*_{n-1,j}$ is either the non-separating fixed point $\beta_{n-1,j}$ or the separating fixed point $\alpha_{n-1,j}$,
%(iii) arguments of the points of $\Gamma_n(x)$ lie in one window of $s_{n,j}$ while arguments of the points of $\tilde\Gamma_n(x)$ lie in another window of $s_{n,j}$,
(iii) start points of $\Gamma^{1/2}_{n}(x_j), \tilde\Gamma^{1/2}_{n}(x_j)$
have equal Green level which is bigger than $\tilde C_3$,
(iv) $x_{j}-\beta^*_{n-1,j}\to 0$ as $n\to\infty$ uniformly in $j$ and $x_j$.
We add one more property as follows. Let
$$\Gamma_{n,j}=\Gamma^{1/2}_{n}(x_j)\cup \beta^*_{n-1,j}\cup\tilde\Gamma^{1/2}_{n}(x_j).$$
Then: (v) $\Gamma_{n,j}$ is a simple curve; the level of $z\in \Gamma_{n,j}\setminus\{\beta^*_{n-1,j}\}$ is positive and decreases (not strickly) from
$\tilde C_3$ to zero along $\Gamma^{1/2}_{n}(x_j)$ and then increases from zero to $\tilde C_3$ along $\tilde\Gamma^{1/2}_{n}(x_j)$;
moreover, if $j_1,j_2\in L_n$, $j_1\neq j_2$, then
$\Gamma_{n,1}$, $\Gamma_{n,j_2}$ are either disjoint or meet at the unique common point
$\beta_{n-1,j_1}=\beta_{n-1,j_2}$ and then disjoint with all others $\gamma_{n-1,j}$, $j\neq j_1,j_2$.
This is because, by property (i), $\Gamma_{n,j}\subset\overline{U_{n-1,j}}$ where (by (C), Sect \ref{prel}) any two $\overline{U_{n-1,j}}$, $\overline{U_{n-1,\tilde j}}$, $j\neq\tilde j$, are
either disjoint or meet at $\beta:=\beta_{n-1,j}=\beta_{n-1,\tilde j}$ in which case $f^{p_{n-1}}$ is satellite. In the considered case, any satellite is doubling
so $\beta\neq\beta_{n-1,i}$ for all $i$ different from $j,\tilde j$.

We assign, for the use below, a 'small' Julia set $I_{n,j}$ to each $\Gamma_{n,j}$ as follows: by the construction,
$\beta^*_{n-1,j}$ is either the $\beta$-fixed point of $f^j(J_{n-1})$, or the $\alpha$-fixed point of $f^j(J_{n-1})$.
In the former case, let $I_{n,j}=f^j(J_{n-1})$, and in the latter case, $I_{n,j}=f^j(J_n)$ (one of the two small Julai sets of the next level $n$ that
are contained in $f^j(J_{n-1})$. Observe that $I_{n,j}\cap\Gamma_{n-1,j}=\{\beta^*_{n-1,j}\}$ and is disjoint
with any other $\Gamma_{n, j'}$ provided $\Gamma_{n,j}$, $\Gamma_{n,j'}$ are disjoint.

There are two subcases B1-B2 to distinguish depending on whether arguments of end points of $\Gamma_{m,j}$ become close or not.
If yes, then one can join the end points of some $\Gamma_{n,j}$ by an arc of equipotential inside of $B(x_j,2\rho)\supset\Gamma_{m,j}$ to surround
a small Julia set as in case A, which would lead to a contradiction. If no, the construction is more subtle: we build a domain ('quadrilateral') in $B(x_j,2\rho)$ bounded by two disjoint
curves as above completed by two arcs of equipotential that join ends of different curve, so that the obtained quadrilateral again contains a small Julia set.

{\bf B1}: $\liminf_{n\in\mathcal{S},j\in L_n}|S_{n-1,j}|<\delta$.

By property (i) listed above and the definition of $\delta$, there are a sequence $(n_k)\subset\mathcal{S}$, $j_k\in L_{n_k}$
and $x_{j_k}$ as above,
such that two ends of each curve $\Gamma_{n_k,j_k}$ can be joined
inside of $B(x_{j_k},\rho)$ by an arc $A^k$ of equipotential of fixed level $\tilde C_3$ such that all arguments of points in $A^k$ belong to $S_{n_k-1,j_k}$. Then we arrive at a contradiction as in case A.

\begin{figure}[h]
\includegraphics[width=12cm]{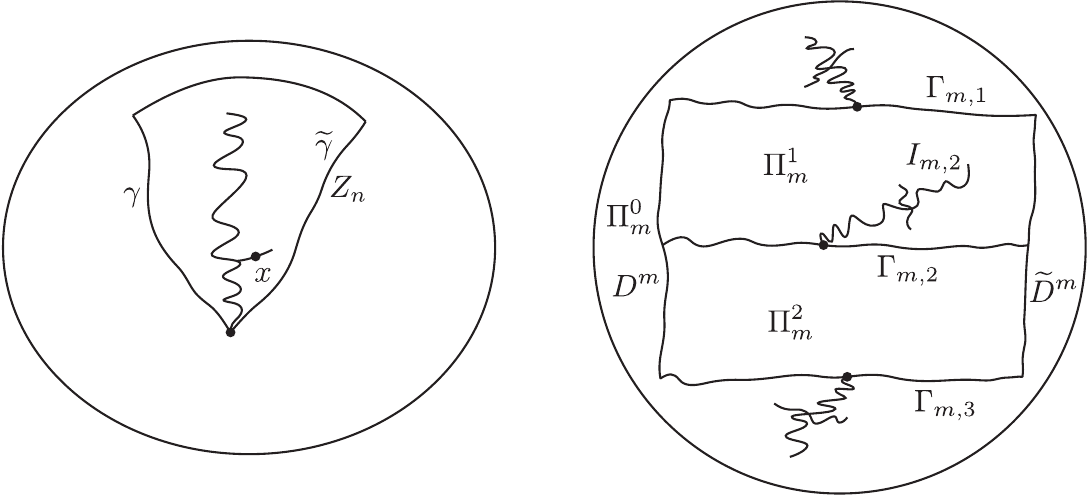}
\caption{Left: Case A and Case B1, right: Case B2.}
\label{fig1}
\end{figure}

%each pair of curves $\Gamma^{1/2}_{n_k}(x_{j_k}), \tilde\Gamma^{1/2}_{n_k}(x_{j_k})$ can be joined
%inside of $B(x_{j_k},\rho)$ by an arc $A^k$ of equipotential of fixed level $\tilde C_3$ such that all arguments of points in $A^k$ belong to $S_{n_k-1,j_k}$. Then we arrive at a contradiction as in case A.

{\bf B2}: $|S_{n-1,j}|\ge \delta$ {\it for all big $n\in\mathcal{S}$ and all $j\in L_n$}.

Fix $n,m\in\mathcal{S}$, $m-n\ge 3$. Define a subset of $L_n$ as follows:
$$L_{n}^m=\{0<j<p_{n-1}| \mu(f^j(J_{n-1})\cap (\tilde Y_{n-1}\cap\tilde Y_{m-1}))>\frac{1-2^{12}\epsilon}{p_{n-1}}\}.$$
As $\mu(\tilde Y_{n-1}\cap\tilde Y_{m-1})>1-12\epsilon$,
$$\# L_{n}^m>(1-3/2^{10})p_{n-1}.$$
For each $j\in L_n^m$ we define further
$$L_{n,j}^m=\{0<k<p_{n-1}| f^k(J_{m-1})\subset $$
$$f^j(J_{n-1}) , \mu(f^k(J_{m-1})\cap(\tilde Y_{n-1}\cap\tilde Y_{m-1}))>\frac{1-2^{16}\epsilon}{p_{m-1}}\}.$$

%Then
%$$\# L_{n,j}^m\ge 2$$
%as otherwise $\# L_{n,j}^m=1$ and,
% therefore,
%$(1-2^{12}\epsilon)/p_{n-1}<1/p_{m-1}+(p_{m-1}/p_{n-1}-1)(1-2^{16}\epsilon)/p_{m-1}=2^{16}\epsilon/p_{m-1}+(1-2^{16}\epsilon)/p_{n-1}$,
%i.e., $p_{m-1}/p_{n-1}<2^{16}\epsilon/(2^{16}\epsilon-2^{12}\epsilon)=1/(1-2^{-4})<2$, a contradiction.

Then
$$\# L_{n,j}^m\ge 5$$
as otherwise $\# L_{n,j}^m\le 4$ and,
 therefore,
$(1-2^{12}\epsilon)/p_{n-1}<4/p_{m-1}+(p_{m-1}/p_{n-1}-4)(1-2^{16}\epsilon)/p_{m-1}=2^{18}\epsilon/p_{m-1}+(1-2^{16}\epsilon)/p_{n-1}$,
i.e., $p_{m-1}/p_{n-1}<2^{18}\epsilon/(2^{16}\epsilon-2^{12}\epsilon)=4/(1-2^{-4})<8$, a contradiction because
$p_{m-1}/p_{n-1}\ge 2^{m-n}\ge 2^3$.

Fix $j\in L_n^m$. Thus $L_{n,j}^m$ contains $5$ pairwise different indices $k_i$, $1\le k\le 5$. As $L_{n,j}^m\subset L_m$, we find
$5$ curves $\Gamma_{m-1, k_i}$. By property (v), if two of them meet, they are disjoint with all others.
Therefore, there are at least $3$ of them denoted by $\Gamma_{m-1,r_i}$, $i=1,2,3$, which are pairwise disjoint.
Let $w_i,\tilde w_{m,i}$ be two ends of $\Gamma_{m-1,r_i}$.

For each $i=1,2,3$, arguments of points of $w_{m,i},\tilde w_{m,i}$ lie in different 'windows' of $s_{m-1,r_i}$.
On the other hand, by the choice of $j$, $s_{m-1,r_i}\subset s_{n-1,j}\subset S_{n-1,j}$. As $n$ is big enough,
lengths of 'windows' of $s_{n-1,j}$ are less than $\delta$. But since we are in case B2, the length of $S_{n-1,j}$ is bigger than $\delta$.
One can assume, therefore, that, for $i=1,2,3$, arguments of $w_{m,i}$ lie in one window of $s_{n-1,j}$ while arguments of $\tilde w_{m,i}$ are in another window.
Therefore, differences of arguments of all $w_{m,i}$ tend to zero as $m\to\infty$, and the same for $\tilde w_{m,i}$.
As all $w_{m,i},\tilde w_{m,i}\in E_\epsilon$, this implies by Lemma \ref{p0} that $\max_{1\le i,l\le 3}|w_{m,i}-w_{m,l}|\to 0$. This along with property (iv) implies that
$\gamma_{m-1,r_i}\subset B(w_{m,1}, 2\rho)$, $i=1,2,3$, for all big $m$. Since, for big $m$, differences of arguments of all $w_{m,i}$ are less than $\delta$,
and the same for $\tilde w_{m,i}$, one can joint all $w_{m,i}$ by an arc $D^m$ of equipotential of level $\tilde C_3$ and all $\tilde w_{m,i}$ by an arc $\tilde D^m$ of equipotential of the same level $\tilde C_3$
such that $D^m,\tilde D^m\subset B(w_1,2\rho)$. Let the end points of $D^m$ be, say, $w_{m,1}$ and $w_{m,3}$, so that $w_{m,2}\in D^m$ is in between.
Since all $3$ curves $\Gamma_{m-1,r_i}$, $i=1,2,3$, are pairwise disjoint, the end points of $\tilde D^m$ have to be then $\tilde w_{m,1}$ and $\tilde w_{m,3}$, so that $\tilde w_{m,2}\in \tilde D^m$ is in between.
Therefore, we get a 'big' quadrilateral $\Pi_m^0\subset B(w_{m,1}, 2\rho)$ bounded by $D^m,\tilde D^m, \Gamma_{m,1}, \Gamma_{m,3}$ where $\Gamma_{m,i}:=\Gamma_{m-1,r_i}$, $i=1,2,3$. The curve $\Gamma_{m,2}$ splits $\Pi_m^0$ into two
'small' quadrilaterals $\Pi_m^1,\Pi_m^2$ with a common curve $\Gamma_{m,2}$ in their boundaries. Recall now that the curve $\Gamma_{m,2}$ comes with a small Julia set $I_{m,2}$ of level either
$m-1$ or $m$, such that $I_{m,2}\cap \Gamma_{m,2}$ is a single point while $I_{m,2}$ is disjoint with $\Gamma_{m,1}$, $\Gamma_{m,3}$. Therefore,
$I_{m,2}\subset\Pi_m^0\subset B(w_{m,1}, 2\rho)$, a contradiction with Remark \ref{genrem}.

\section{Proof of Corollaries \ref{coro:J}-\ref{expat0}}\label{proofJ}
Corollary \ref{coro:J} follows directly from the following
\begin{prop}\label{pr}
Let $f$ be an infinitely renormalizable quadratic polynomial. Then conditions (1)-(4) are equivalent:
\begin{enumerate}
\item $f: J_\infty\to J_\infty$ has no invariant probability measure with positive exponent,
\item for every neighborhood $W$ of $P$ and every $\alpha\in (0,1)$ there exist $m_0$ and $n_0$ such that, for each
$m\ge m_0$ and $x\in orb(J_n)$ with $n\ge n_0$,
$$\frac{\#\{i| 0\le i<m, f^i(x)\in W\}}{m}>\alpha;$$
additionally, $f: P\to P$ has no invariant probability measure with positive exponent,
%any invariant probability measure of $f: P\to P$ has a non-positive exponent,
\item every invariant probability measure of $f: J_\infty\to J_\infty$ is, in fact, supported on $P$ and has zero exponent,
\item for every invariant probability ergodic measure $\mu$ of $f$ on the Julia set $J$ of $f$, either
$\supp(\mu)\cap J_\infty=\emptyset$ and its Lyapunov exponent $\chi(\mu)>0$,
or $\supp(\mu)\subset P$ and $\chi(\mu)=0$.
\end{enumerate}
\end{prop}
\begin{proof}
%It is based on the fact that $f$ expands the hyperbolic metric on $\C\setminus P$: for
%any compactly contained in $\C\setminus P$ open set $U$ there is $\lambda>1$ such
%that $|Df^k(y)|>\lambda$ whenever $y, f^k(y)\in U$, where $D$ the derivative
%w.r.t. the hyperbolic metric of $U$.
(1)$\Rightarrow$(2). Assume the contrary.
Let $E=\C\setminus W$. Since $W$ is a neighborhood of a compact set $P$, the Euclidean distance $\dist(E, P)>0$.
By a standard normality argument, as all periodic points of $f$ are repelling,
there are $\lambda>1$ and $k_0>0$ such that $|(f^k)'(y)|>\lambda$ whenever $y, f^k(y)\in E$
and $k\ge k_0$.
As (2) does not hold, find $\alpha\in (0,1)$, a sequence $n_k\to\infty$, points
$x_k\in orb(J_{n_k})$ and a sequence $m_k\to\infty$ such that, for each $k$,
$$\frac{\#\{i: 0\le i<m_k, f^i(x_k)\in E\}}{m_k}\ge\beta:=1-\alpha.$$
Fix a big $k$ such that $\beta m_k>3k_0$ and consider the times $0\le i_1^k<i_2^k<...i_{l_k}^k<m_k$ where $l_k/m_k\ge \beta$ such
that $f^i(x_k)\in E$. Let $z_k=f^{i_1^k}(x_k)$ so that $z_k\in E\cap orb(J_n)$. Therefore, by the choice of $\lambda$ and $k_0$,
$|(f^{m_k-i_1^k})'(z_k)|\ge \tilde\lambda^{m_k}\ge \tilde\lambda^{m_k-i_1}$
where $\tilde\lambda=\lambda^{\frac{\beta}{2k_0}}>1$.
%Note that $m_k-i_1^k\to\infty$.
%\lambda^l\ge \tilde\lambda^{m_k}$
%for some $\lambda>1$ that depends only on the domain $W$,
%$|Df^{m_k}(x)|\ge \lambda^l\ge \tilde\lambda^{m_k}$
%where $\tilde\lambda=\lambda^\beta>1$.
In this way we get a sequence of measures
$\mu_k=\frac{1}{m_k-i_1^k}\sum_{i=0}^{m_k-i_1^k-1}\delta_{f^i(z_k)}$
such that the Lyapunov exponent of $\mu_k$ is at
least $\log\tilde\lambda>0$. Passing to a subsequence one can assume that $\{\mu_k\}$
converges weak-* to a measure $\mu$. Then $\mu$ is an $f$-invariant probability measure on $J_\infty=\cap orb(J_n)$ with the exponent at least $\log\tilde\lambda > 0$, a contradiction
with (1).

(2)$\Rightarrow$(3), by the Birkhoff Ergodic Theorem along with \cite{Pr0}.

(3)$\Rightarrow$(4): let $\mu$ be as in (4) and $\overline{U}\cap P=\emptyset$ for some open set $U$ with $\mu(U)>0$.
Let $F:U\to U$ be the first return map equipped with the induced invariant measure $\mu_U$.
By the Birkhoff Ergodic Theorem and by an argument as in (1)$\Rightarrow$(2), the exponent $\chi_F(\mu_U)$ of $F$ w.r.t. $\mu_U$ is strictly positive.
Hence, $\chi(\mu)=\mu(U)\chi_F(\mu_U)$ is positive too. This proves the implication.

And (4) obviously implies (1).
\end{proof}

\begin{proof}[Proof of Corollary \ref{expat0}]
If $\overline\chi(x)$ were strictly positive, for some $x\in J_\infty$, that would imply, by a standard argument (see the proof of Corollary \ref{coro:J}), the existence
of an $f$-invariant measure with positive exponent supported in $\omega(x)\subset J_\infty$,
with a contradiction to Theorem \ref{thm:meas+}. This proves (\ref{expany}).
By \cite{LPS}, $\liminf_{n\to \infty} \frac{1}{n}\log|(f^n)'(c)|\ge 0$. On the other hand, by (\ref{expany}),
$\overline\chi(c)\le 0$, which proves (\ref{expcr}).
%On the other hand, if
%$\overline\chi(c)$ was strictly positive that would imply, by a standard argument, the existence
%of an $f$-invariant measure with positive exponent supported in $\omega(c)\subset J_\infty$,
%with a contradiction to Theorem \ref{thm:meas+}. This proves (\ref{expcr}).
%The proof of (\ref{expset}) is similar
%taking into account that $\overline\chi(orb(J_n))>0$ for all $n$.
%Indeed, otherwise there are sequences $p_{n_i}\to\infty$, $x_i\in orb(J_{n_i})$ and $\epsilon>0$ such that $\overline\chi_f(x_i)\ge \epsilon$
%for all $i$. As above, this implies the existence
%of an $f$-invariant measure with exponent at least $\epsilon$ supported in $\cap_i orb(J_{n_i})=J_\infty$, a contradiction.
\end{proof}

\iffalse

\begin{figure}[h]
\includegraphics[width=7cm,height=7cm]{CCF_000008.pdf}
\caption{Case A2}
\label{fig3}
\end{figure}

\fi

\end{document}